\documentclass[11pt,reqno]{amsart}

\usepackage[T1]{fontenc}
\usepackage[utf8]{inputenc}

\usepackage{geometry}
\usepackage{etoolbox}
\expandafter\preto\csname section\endcsname{\par\vspace{0.5\baselineskip}}
\expandafter\preto\csname subsection\endcsname{\par\vspace{0.3\baselineskip}}
\expandafter\preto\csname subsubsection\endcsname{\par\vspace{0.2\baselineskip}}

\numberwithin{equation}{section}

\usepackage{amsmath, amssymb, amsfonts, mathrsfs}
\usepackage{mathtools}
\usepackage{enumitem}

\usepackage{graphicx}
\usepackage[all]{xy}
\usepackage{color}

\usepackage{caption}
\usepackage{comment}

\usepackage{hyperref}
\usepackage{bookmark}
\hypersetup{
    unicode            = true,
    bookmarksnumbered  = true,
    colorlinks         = false,
    hidelinks,
    final
}
\usepackage{xurl}
\usepackage{placeins}

\theoremstyle{plain}
\newtheorem{theorem}{Theorem}[section]
\newtheorem{lemma}[theorem]{Lemma}
\newtheorem{proposition}[theorem]{Proposition}
\newtheorem{corollary}[theorem]{Corollary}

\theoremstyle{definition}
\newtheorem{definition}[theorem]{Definition}
\newtheorem{example}[theorem]{Example}
\newtheorem{notation}[theorem]{Notation}

\newtheorem{assumption}[theorem]{Assumption}

\newtheorem{question}[theorem]{Question}

\theoremstyle{remark}
\newtheorem{remark}[theorem]{Remark}

\newcommand{\Ext}{\operatorname{Ext}}

\begin{document}

\title{
Principal twistor models and asymptotic hyperkähler metrics
}

\author{Ryota Kotani}
\address{Department of Mathematics, Institute of Science Tokyo}
\email{kotani.r.597b@m.isct.ac.jp, kotani.r.337@gmail.com}

\date{\today}

\begin{abstract}
Let $X$ be a conical symplectic variety admitting a crepant resolution $Y$. Based on the theory of universal Poisson deformations, we construct a complex manifold called the \textit{principal twistor model} associated with $Y$. We prove a universality theorem for this model: if the regular locus of $X$ admits an algebraic hyperkähler cone metric $g_0$, then the twistor space of any algebraic hyperkähler metric on $Y$ asymptotic to $g_0$ is uniquely recovered by slicing the principal twistor model. As an application, we use this universality to study the moduli space of algebraic hyperkähler structures with asymptotic behavior. We show that this moduli space admits an inclusion into a finite-dimensional real vector space, and furthermore, that this inclusion is an open embedding when $X$ has an isolated singularity.
\end{abstract}

\maketitle

\section{Introduction}

In various fields of mathematics and theoretical physics, hyperkähler metrics asymptotic to a hyperkähler cone appear naturally as geometric objects linking moduli theory, representation theory, and algebraic geometry, and significant progress has been made in the systematic research on their construction and classification.
For the construction of such metrics, there are three complementary perspectives: a differential-geometric approach based on nonlinear partial differential equations (PDEs), a complex symplectic geometric framework using twistor spaces, and an algebraic-geometric construction based on hyperkähler quotients. Typical examples of these metrics include ALE gravitational instantons \cite{Hitchin79}, QALE hyperkähler metrics \cite{Joyce01} as their higher-dimensional generalizations, and metrics appearing on toric hyperkähler manifolds \cite{Bielawski_Dancer00} and Nakajima quiver varieties \cite{Nakajima94}.

A hyperkähler metric is defined as a metric that is Kähler with respect to three complex structures satisfying the quaternionic relations. In this case, we naturally obtain a deformation of holomorphic symplectic manifolds indexed by the Riemann sphere $\mathbb{P}^1$, endowed with a real structure and a family of rational curves called twistor lines; this collection of data is referred to as the twistor space. Conversely, when a twistor space is given, the hyperkähler structure---both the metric and the triple of complex structures---can be reconstructed from it \cite{HKLR}.
From this viewpoint, Hitchin \cite{Hitchin79} constructed ALE gravitational instantons, which are typical examples of algebraic hyperkähler metrics with asymptotic behavior.

In general, it is difficult to determine whether a rational curve on a deformation is a twistor line, and thus twistor lines cannot be handled directly.
On the other hand, it is known that the structure of a space admitting a hyperkähler metric with asymptotic behavior is largely determined by the properties of the asymptotic cone metric.
Motivated by this observation, the goal of this paper is to reveal the structure of the twistor space corresponding to algebraic hyperkähler cone metrics and, based on this, to characterize the structure of the twistor space of algebraic hyperkähler metrics with asymptotic behavior (Theorem~\ref{intro main thm: universality of PTM}).

\medskip
To explain the main results, we first introduce the following notions and setup.

We define a twistor model to be a structure satisfying all conditions of a twistor space except for the existence of a family of twistor lines. As a generalization of twistor models, we define a new notion called the principal twistor model, which consists of a deformation of holomorphic symplectic manifolds indexed by a vector bundle $\mathcal{C}(2) \coloneqq \mathcal{C} \otimes \mathcal{O}(2)$ over $\mathbb{P}^1$ and a real structure on it. For any choice of a real section $s$ of the vector bundle $\mathcal{C}(2)$ (cf.~Def.~\ref{def:real section of C(2)}), we obtain a twistor model $Z_s$ by slicing the principal twistor model along the section $s$ (cf.~Prop.~\ref{prop:slicing of PT}).
In other words, the twistor model $Z_s$ is given by restricting the principal twistor model to the real section $s$.
As will be explained later, the principal twistor model provides a natural framework for studying the moduli space of algebraic hyperkähler structures asymptotic to an algebraic hyperkähler cone (Cor.~\ref{intro cor:inclusion of moduli space}). In this framework, the asymptotic behavior of the metric can be interpreted as a degeneration of twistor spaces (cf.~Remark~\ref{rem:degeneration of tw sp}).

In this paper, as the asymptotic hyperkähler cone, we consider a conical symplectic variety $X$ with a good triple (cf.~Def.~\ref{def: good triple on X}) ---that is, an affine variety equipped with a good $\mathbb{C}^*$-action, a holomorphic symplectic form, and a quaternionic structure.
As will be explained later, this condition is necessary for a conical symplectic variety to admit an algebraic hyperkähler cone metric (Theorem~\ref{intro main thm: universality of PTM}). We further assume that the variety $X$ admits
a projective crepant resolution $Y$ (which we simply call a crepant resolution).

To study the geometry of $Y$, we use the theory of Poisson deformations. A Poisson deformation is a deformation of a complex manifold equipped with a compatible deformation of the Poisson bracket, which corresponds to a relative holomorphic symplectic form. By a theorem of Namikawa \cite{Namikawa11}, the crepant resolution $Y$ admits the universal Poisson deformation $\mathcal{Y} \to \mathcal{C}$, which captures all such deformations, and its base space $\mathcal{C}$ is naturally isomorphic to $H^2(Y;\mathbb{C})$.
Based on the idea of Bielawski and Foscolo \cite{Bielawski-Foscolo21} of gluing this universal Poisson deformation via the $\mathbb{C}^*$-action, the principal twistor model $\mathcal{Y}(1)$ is constructed (Prop.~\ref{intro prop: construction of PTM}; see Figures~\ref{fig:PTM} and \ref{fig:twistor model in PTM} on page~\pageref{fig:PTM}).

\medskip
We present three main results of this paper. First, we show that the principal twistor model is naturally constructed from a good triple on a conical symplectic variety. This result is shown by arguments on formal universal Poisson deformations and applying a $\mathbb{C}^*$-gluing construction (cf.~\S \ref{subsec:C^*-gluing construction}):

\begin{proposition}[Proposition \ref{main prop: PTM of Y}]\label{intro prop: construction of PTM}

    Let $X$ be a conical symplectic variety with a good triple, and assume that it admits a crepant resolution $Y$. Let $\mathcal{Y} \rightarrow \mathcal{C}$ denote the universal Poisson deformation of $Y$, where the base space $\mathcal{C}$ is isomorphic to $H^2(Y;\mathbb{C})$.
    Then, the good triple naturally extends to a triple on $\mathcal{Y}$. Furthermore, by applying the $\mathbb{C}^*$-gluing construction to this triple, the principal twistor model $\mathcal{Y}(1) \rightarrow \mathcal{C}(2)$ is constructed.
\end{proposition}

\begin{remark}
    The extension of the data other than the quaternionic structure to the universal Poisson deformation is due to Namikawa \cite{Namikawa08, namikawa2026notes}. In this work, we show that the quaternionic structure can be similarly extended.
\end{remark}

The following theorem on the universality of the principal twistor model is the central result of this paper. This theorem states that any twistor space corresponding to an algebraic hyperkähler metric asymptotic to a hyperkähler cone is uniquely obtained by slicing the principal twistor model along a real section. This result is shown by arguments on the universality of the universal Poisson deformation and the asymptotic behavior of the metric:

\begin{theorem}[Theorem~\ref{main thm: universality of PTM}]\label{intro main thm: universality of PTM}
    Let $X$ be a conical symplectic variety, and let $g_0$ be an algebraic (cf.~Def.~\ref{def:alg for tw and HK met}) hyperkähler cone metric (cf.~Def.~\ref{def:cone met on coni symp vrt}) on its regular locus $X_{\mathrm{reg}}$. Let $Y$ be a crepant resolution of $X$.
    \begin{enumerate}
        \item The metric $g_0$ naturally defines a good triple on $X$. Furthermore, for this good triple, the principal twistor model $\mathcal{Y}(1) \rightarrow \mathcal{C}(2)$ is constructed by Proposition \ref{intro prop: construction of PTM}.

        \item Consider an algebraic hyperkähler structure $(Y, g, I, J, K)$ asymptotic to the structure corresponding to $g_0$ (cf.~Def.~\ref{def:alg asymp HK on Y}). Then, there exists a unique real section $s$ of the vector bundle $\mathcal{C}(2)$ such that the twistor space $Z$ corresponding to the metric $g$ is isomorphic, as a twistor model, to the model $Z_s$ obtained by slicing the principal twistor model $\mathcal{Y}(1)$ along the real section $s$.
    \end{enumerate}
\end{theorem}

\begin{remark}
    This theorem implies that the principal twistor model determines all the geometric data of the corresponding twistor space, except for the twistor lines.
    For instance, if we apply Proposition~\ref{intro prop: construction of PTM} and Theorem~\ref{intro main thm: universality of PTM} to the quotient space $X=\mathbb{C}^2/\Gamma$ ($\Gamma < SU(2)$), then we obtain the twistor models corresponding to ALE gravitational instantons.
\end{remark}

Finally, as a corollary of Theorem~\ref{intro main thm: universality of PTM}, we present an application to the moduli space of algebraic hyperkähler structures with asymptotic behavior. This result is obtained by combining an analysis of twistor lines with the asymptotic behavior of the metric:

\begin{corollary}[Cor.~\ref{cor:inclusion of moduli space}, Cor.~\ref{cor:moduli space of HK str, isolated sing case}]\label{intro cor:inclusion of moduli space}
    Let $X$ be a conical symplectic variety. Assume that the regular locus $X_{\mathrm{reg}}$ admits an algebraic hyperkähler cone metric $g_0$, and that $X$ has a crepant resolution $Y$.
    Let $\mathcal{M}$ be the moduli space of algebraic hyperkähler structures $(Y,g,I,J,K)$ defined as follows:
    \[ \mathcal{M} = \left\{ (Y,g,I,J,K) \mid \text{the metric $g$ is asymptotic at infinity to $g_0$} \right\} / (\text{isomorphisms}). \]
    Then, the moduli space $\mathcal{M}$ admits an inclusion into the real vector space $H^2(Y;\mathbb{R})\otimes \mathbb{R}^3$.
    Furthermore, if $X$ has an isolated singularity, then this inclusion is open. Thus, if $\mathcal{M}$ is non-empty, then we have $\dim_{\mathbb{R}}\mathcal{M}=3d$, where $d=\dim H^2(Y;\mathbb{C})$.
\end{corollary}

\begin{remark}
Corollary \ref{intro cor:inclusion of moduli space} can be viewed as an analogue of the injectivity part of the Torelli-type theorem for ALE gravitational instantons by Kronheimer \cite{Kronheimer89}. The geometric relationship between this inclusion and the periods of the hyperkähler structures is briefly discussed in Remark~\ref{rem:Torelli-type injectivity}.
\end{remark}

We now clarify the relationship with the prior work of Bielawski and Foscolo \cite{Bielawski-Foscolo21}. Bielawski and Foscolo introduced an idea to construct a space similar to the singular model of the principal twistor model (cf.~Remark~\ref{remark:sing PTM of X}) using the universal Poisson deformation of the space $X$ before taking the crepant resolution, and then to construct a singular model of the twistor space by slicing. This work builds on this idea to construct the principal twistor model. While \cite{Bielawski-Foscolo21} aims at the construction of (singular) twistor spaces, this work aims at proving the universality of the principal twistor model. Furthermore, there are differences in the setup, such as the good triple on the hyperkähler cone; in this work, we provide a rigorous proof of the extension of the triple to the Poisson deformation and a systematic formulation of the $\mathbb{C}^*$-gluing construction.

\medskip
The organization of the paper is as follows. In \S\ref{sec:PD of good triple vrt}, we introduce conical symplectic varieties with a good triple and show that the triple extends to the universal Poisson deformation of its crepant resolution. In \S\ref{sec:PTM}, we define the principal twistor model and explain its construction after introducing the $\mathbb{C}^*$-gluing construction. We also prove the universality of the principal twistor model. In \S\ref{sec:Applications}, as an application of this work, we study the moduli space of algebraic hyperkähler structures with asymptotic behavior. We also introduce metrics constructed via hyperkähler quotients and QALE hyperkähler metrics as concrete examples to which this work applies under certain assumptions. Finally, we conclude the paper by discussing future directions.

\subsection*{Convention}
Throughout this paper, a ``hyperkähler metric'' $g$ on $Y$ often implicitly refers to a hyperkähler structure $(Y, g, I, J, K)$ with fixed complex structures, so that the corresponding twistor space is uniquely determined.

\section{Conical symplectic varieties with good triples}\label{sec:PD of good triple vrt}

In this section, as a preparation for stating the main results, we introduce conical symplectic varieties with a good triple. We also show that when such a variety admits a crepant resolution, the good triple naturally lifts to the universal Poisson deformation of the resolution.

\subsection{Definitions}\label{subsec:def: good triple on X}

In this subsection, we introduce the definition and examples of conical symplectic varieties with a good triple.
\medskip

First, we recall the definition of a conical symplectic variety.

\begin{definition}[cf.~\cite{Namikawa15}] \label{def: coni symp}
    We call an affine holomorphic symplectic variety $(X, \omega)$ equipped with a good $\mathbb{C}^*$-action $\lambda$ a \textit{conical symplectic variety}. Here, a \textit{good $\mathbb{C}^*$-action} is a $\mathbb{C}^*$-action $\lambda$ satisfying the following two conditions:
    \begin{enumerate}
     \item The coordinate ring $\Gamma(X,\mathcal{O}_X)$ is positively graded with respect to the action $\lambda$.
     \item The action $\lambda$ acts on the holomorphic symplectic $2$-form $\omega$ with positive weight.
    \end{enumerate}
\end{definition}

\begin{remark}
    A conical symplectic variety $X$ may be singular. In this case, we consider $\omega$ as a $2$-form on the regular locus $X_{\mathrm{reg}}$.
\end{remark}

Next, we define a quaternionic structure on a complex variety equipped with a $\mathbb{C}^*$-action.

\begin{definition}\label{def:quat str}
    An anti-holomorphic automorphism $j$ on a complex variety $X$ with a $\mathbb{C}^*$-action $\lambda$ is called a \textit{quaternionic structure} on $X$ if it satisfies
    \[j^2 = \lambda_{-1},\]
    where $\lambda_{-1}$ denotes the action of $-1 \in \mathbb{C}^*$.
\end{definition}

\begin{remark}
    When $\lambda_{-1}=\lambda_1 = \mathrm{id}$, $j$ is actually a real structure. However, for convenience, we also refer to $j$ as a quaternionic structure on $X$ in this case.
\end{remark}

We now introduce the notion of a good triple on a conical symplectic variety.

\begin{definition} \label{def: good triple on X}
    Let $(X, \omega)$ be a conical symplectic variety with $\mathbb{C}^*$-action $\lambda$, and let $j$ be a quaternionic structure on $X$. We call the triple $(\lambda, \omega, j)$ a \textit{good triple} on $X$ if they are pairwise compatible in the following sense: for all $t \in \mathbb{C}^*$,
    \begin{align*}
       \text{(1)  }&\lambda_t^*\omega = t^2\omega, \\
       \text{(2)  }&j^*\omega=\overline{\omega}, \\
       \text{(3)  }&j\circ \lambda_t=\bar{\lambda}_t\circ j.
    \end{align*}
\end{definition}

\begin{remark}
    We only consider the case where the weight of the $\mathbb{C}^*$-action on the symplectic $2$-form $\omega$ is $2$.
\end{remark}

\begin{remark}
    As will be shown later in Theorem~\ref{main thm: universality of PTM}, admitting a good triple is in fact a necessary condition for a conical symplectic variety to admit an algebraic hyperkähler cone metric. This geometric requirement motivates the above definition.
\end{remark}

\begin{example}[Quotient spaces] \label{ex:std triple}
    \begin{enumerate}
       \item The vector space $\mathbb{C}^{2n}$ admits a standard good triple as follows:
       in terms of standard coordinates $(x_1,\ldots,x_n;y_1,\ldots, y_n)$,
        \begin{align*}
        \text{(a)  }&\lambda=\lambda_{\text{Scal}}\ (\text{scalar multiplication}), \\
        \text{(b)  }&\omega=dx_1\wedge dy_1 + \cdots + dx_n\wedge dy_n, \\
        \text{(c)  }&j(x_i,y_i)=(\bar{y}_i,-\bar{x}_i),\ (i=1,\ldots,n).
        \end{align*}

       \item Let $G$ be a finite subgroup of the compact symplectic group $Sp(n)$ whose action preserves $\omega$. Consider the quotient space $X=\mathbb{C}^{2n}/G$. Then, the action of any element $g\in G$ is compatible with the standard good triple $(\lambda_{\text{Scal}}, \omega, j)$ on $\mathbb{C}^{2n}$ as follows:
             \begin{align*}
                \text{(a)  }&\lambda\circ g=g\circ \lambda, \\
                \text{(b)  }&g^*\omega=\omega, \\
                \text{(c)  }&j\circ g= g \circ j.
             \end{align*}
             Therefore, $X$ naturally inherits the good triple from the standard one on $\mathbb{C}^{2n}$.
    \end{enumerate}
\end{example}

\begin{example}[Nilpotent cone] \label{ex:nilpotent cone}
    Let $G$ be a complex semisimple Lie group and $\mathfrak{g}$ its Lie algebra. Let $\mathcal{N}_G$ be the nilpotent cone with respect to the adjoint action on $\mathfrak{g}$.
    Let $\lambda_{\text{Scal}}$ be the scalar multiplication on $\mathfrak{g}$, $\omega_{KK}$ the Kirillov-Kostant form, and $\tau$ a Cartan involution (cf.~\cite{Collingwood-McGovern93}). Then $\lambda_{\text{Scal}}$ and $\tau$ preserve the nilpotent cone $\mathcal{N}_G$, and the restriction of the triple $(\lambda_{\text{Scal}}^2, \omega_{KK}, \tau)$ to $\mathcal{N}_G$ defines a good triple on the conical symplectic variety $\mathcal{N}_G$
    (cf.~\cite[\S 3.1]{Bielawski-Foscolo21}, \cite{Kronheimer90}).
\end{example}

\subsection{Universal Poisson deformations}

In this subsection, we show that when a conical symplectic variety with a good triple admits a crepant resolution, the good triple naturally extends to the universal Poisson deformation of the resolution.

\medskip

Before stating the results, we briefly recall the notion of Poisson deformations. A Poisson deformation of a holomorphic symplectic variety is a deformation of the complex space equipped with a compatible deformation of the Poisson structure, which corresponds to a relative holomorphic symplectic form. A universal Poisson deformation is characterized by its universal property: any Poisson deformation can be obtained uniquely as a pullback from this family. In the proofs later in this section, we also use formal Poisson deformations, which correspond to successive extensions of infinitesimal deformations.

\medskip
First, we recall results by Namikawa on universal Poisson deformations.

\begin{theorem}[Namikawa {\cite[\S 5]{Namikawa11}, \cite{Namikawa08, Namikawa15}} ] \label{Namikawa Thm:CM and univ Poisson}
    Let $X$ be a conical symplectic variety.
    Assume that $X$ admits a crepant resolution $Y$. Then, there exist universal Poisson deformations $\mathcal{X}\rightarrow\mathcal{B}$ and $\mathcal{Y}\rightarrow\mathcal{C}$ of $X$ and $Y$, respectively. Here, the base spaces $\mathcal{B}$ and $\mathcal{C}$ are vector spaces, and $\mathcal{C}$ is naturally isomorphic to $H^2(Y;\mathbb{C})$ via the period map.
    Furthermore, there exists a quotient map $q:\mathcal{C}\rightarrow\mathcal{B}$ by a finite Galois group such that if we let $\mathcal{X}'\rightarrow\mathcal{C}$ be the pullback of $\mathcal{X}$ via $q$, then $\mathcal{Y}$ is a simultaneous resolution of $\mathcal{X}'$.
    That is, there exists a morphism $\nu:\mathcal{Y}\rightarrow\mathcal{X}'$ such that the following diagram commutes:
    \begin{equation}
    \vcenter{
    \xymatrix{
       \mathcal{Y} \ar[d]_{} \ar[r]^\nu &\mathcal{X}' \ar[d]^{}\ar[r] &\mathcal{X} \ar[d]^{} \\
       \mathcal{C} \ar@{=}[r] &\mathcal{C} \ar[r]^{q}&\mathcal{B} \\
    }} \label{diagram:Yuni, Xuni}
    \end{equation}
    Moreover, for any $c \in \mathcal{C}$, the restriction of the morphism $\nu$ to the fiber over $c$,
    \[
    \nu_c:\mathcal{Y}_{c}\rightarrow \mathcal{X}'_{c}=\mathcal{X}_{q(c)},
    \]
    is a crepant resolution of $\mathcal{X}'_{c}$, and for a generic $c\in\mathcal{C}$, $\nu_c$ is an isomorphism.
\end{theorem}

\begin{remark}\label{rem:affinization of Ycal}
    We refer to the Poisson deformation $\mathcal{X}'$ as the \textit{affinization} of the universal Poisson deformation $\mathcal{Y}$ of $Y$.
\end{remark}

It is known that the good $\mathbb{C}^*$-action on a conical symplectic variety can be extended to the universal Poisson deformation:
\enlargethispage{2\baselineskip} 
\begin{theorem}[Namikawa \cite{namikawa2026notes}, {\cite[Prop.A.7]{Namikawa08}}] \label{Namikawa Thm: C^*-action on Yuni}
    Let $(X,\omega)$ be a conical symplectic variety.
    Assume that $X$ admits a crepant resolution $Y$. Let $\mathcal{X}\rightarrow\mathcal{B}$ and $\mathcal{Y}\rightarrow\mathcal{C}$ be the universal Poisson deformations of $X$ and $Y$, respectively.
    Then, the $\mathbb{C}^*$-action $\lambda$ on $X$ uniquely extends to a $\mathbb{C}^*$-action $\tilde{\lambda}$ on the deformation $\mathcal{X}$. Moreover, this $\mathbb{C}^*$-action $\tilde{\lambda}$ naturally lifts to a $\mathbb{C}^*$-action $\tilde{\lambda}_Y$ on the deformation $\mathcal{Y}$ making the diagram (\ref{diagram:Yuni, Xuni}) commutative.
    Furthermore, the $\mathbb{C}^*$-action induced by $\tilde{\lambda}_Y$ on the base space $\mathcal{C}$ is scalar multiplication of weight $2=\mathrm{wt}(\omega)$.
\end{theorem}

Next, we define a good triple on a Poisson deformation.

\begin{definition} \label{def:good triple on PD}
    Let $X$ be a conical symplectic variety with a good triple.
    Assume that a Poisson deformation $(\mathcal{X},\tilde{\omega})$ of $X$ admits a $\mathbb{C}^*$-action $\tilde{\lambda}$ acting with positive weight and a quaternionic structure $\tilde{j}$.
    We say that the triple $(\tilde{\lambda}, \tilde{\omega}, \tilde{j})$ is a \textit{good triple} on the Poisson deformation $\mathcal{X}$ if it satisfies the following three conditions: for all $t \in \mathbb{C}^*$,
    \begin{align*}
       \text{(1)  }& \tilde{\lambda}_t^*\tilde{\omega} = t^2\tilde{\omega},\\
       \text{(2)  }& \tilde{j}^*\tilde{\omega}=\overline{\tilde{\omega}},\\
       \text{(3)  }& \tilde{j}\circ \tilde{\lambda}_t=\overline{\tilde{\lambda}_t}\circ \tilde{j}.
    \end{align*}
\end{definition}

In this section, we show the following proposition:

\begin{proposition}\label{prop: good triple on Yuni}
    Let $X$ be a conical symplectic variety with a good triple $(\lambda,\omega,j)$. Assume that $X$ admits a crepant resolution $Y$.
    Let $\mathcal{Y}\rightarrow\mathcal{C}$ be the universal Poisson deformation of the crepant resolution $Y$.
    Then, the affinization $\mathcal{X}'\rightarrow\mathcal{C}$ of $\mathcal{Y}$ admits a good triple $(\tilde{\lambda}', \tilde{\omega}', \tilde{j}')$ naturally induced from the good triple on $X$.
    Moreover, this good triple naturally lifts to a unique triple $(\tilde{\lambda}_Y, \tilde{\omega}_Y, \tilde{j}_Y)$ on $\mathcal{Y}$.
\end{proposition}

The fact that the $\mathbb{C}^*$-action extends to the universal Poisson deformation is precisely the statement of Theorem~\ref{Namikawa Thm: C^*-action on Yuni}. In what follows, we show that the quaternionic structure extends to the universal Poisson deformation in a similar manner.

\medskip
First, we show that the quaternionic structure extends to the formal universal Poisson deformation of a conical symplectic variety.
For this proof, we recall the following lemma:

\begin{lemma} \label{lemma:vanishing of group cohomology}
    Let $G$ be a finite group.
    Then, the higher group cohomology of any $\mathbb{C}[G]$-module $M$ vanishes.
    That is, for any $i\in\mathbb{Z}_{>0}$,
    \[H^i(G,M)=0.\]
\end{lemma}

\begin{proof}
    By Maschke's theorem, the group ring $\mathbb{C}[G]$ is semisimple.
    Therefore, any $\mathbb{C}[G]$-module $M$ is a projective module.
    This implies that for any $i\in\mathbb{Z}_{>0}$,
    \begin{equation*}
        H^i(G,M)\simeq \Ext_{\mathbb{C}[G]}^i(\mathbb{C}, M)=0. \qedhere
    \end{equation*}
\end{proof}

We prove the following lemma on the uniqueness of automorphisms on formal Poisson deformations:
\begin{lemma}\label{lemma:uniqueness of f_n^m}
    Let $X$ be a conical symplectic variety.
    Let $\{f_n\}_n$ and $\{f_n'\}_n$ be two Poisson automorphisms on a formal Poisson deformation $\{(X_n\rightarrow S_n,\omega_n)\}_n$ of $X$. Assume they satisfy the following three conditions:
  \begin{enumerate}
   \item They agree on the central fiber $X$. That is, $f_0=f'_0$.
   \item They agree on the base space $S_n$ for each $n$. That is, for the projection $\pi_n:X_n\rightarrow S_n$, we have $\pi_n\circ f_n = \pi_n\circ f_n'$.
   \item There exists a natural number $m$ such that $f_n^m=\mathrm{id}={f_n'}^m$.
  \end{enumerate}
    Then, the two morphisms $\{f_n\}_n$ and $\{f_n'\}_n$ are conjugate.
\end{lemma}

\begin{proof}
    To simplify the notation, we denote the Poisson deformation $(X_n,\omega_n)$ by $\eta_n$. Let
    \[A_n:=\mathrm{PAut}(\eta_n; \mathrm{id}|_{\eta_{n-1}})\]
    be the set of Poisson automorphisms of $\eta_n$ whose restrictions to $\eta_{n-1}$ are the identity.
    This set naturally forms a $\mathbb{C}$-vector space.
    Let $G$ be the cyclic group of order $m$, $\mathbb{Z}/m\mathbb{Z}=\langle\varepsilon:\varepsilon^m=1\rangle$.
    We prove the claim by induction on $n$.
    For $n-1$, assume that there exists $\theta_{n-1}\in A_{n-1}$ such that
    \[f_{n-1} = \theta_{n-1}^{-1} \circ f_{n-1}' \circ \theta_{n-1}.\]
    Then, the group $G$ acts on the vector space $A_n$ by defining ${}^\varepsilon u := f_n\circ u\circ {f_n'}^{-1} \in A_n$ for any $u\in A_n$.
    For $g=\varepsilon^k\in G$, if we set $f_n(g):=f_n^k$ and $f_n'(g) = {f_n'}^k$, then the action of an element $g$ on $u$ can be written as
    \[{}^gu:= f_n(g)\circ u\circ {f_n'}(g)^{-1}.\]
    For any $g\in G$, define an element $u_g \in A_n$ by
    \[u_g:= f_n(g)\circ {f_n'}(g)^{-1}.\]
    Then, $\{u_g\}_{g\in G}$ defines a $1$-cocycle for the group cohomology of the rational representation $A_n$ of $G$. Indeed, for any $g,h\in G$,
  \begin{align*}
    u_{gh} &= f_n(gh) \circ f_n'(gh)^{-1}\\
    &= f_n(g)\circ f_n(h) \circ f_n'(h)^{-1} \circ f_n'(g)^{-1}\\
    &= f_n(g)\circ u_h \circ f_n(g)^{-1}\circ u_g\\
    &= {}^gu_h \circ u_g = u_g \circ {}^gu_h \quad (\text{since } A_n \text{ is abelian}).
  \end{align*}
  Since $H^1(G,A_n)=0$ by Lemma \ref{lemma:vanishing of group cohomology}, $\{u_g\}_{g\in G}$ is a $1$-coboundary. Therefore, there exists an element $\theta_n \in A_n$ such that, in particular,
  $u_\varepsilon = {}^\varepsilon \theta_n \circ \theta_n^{-1}$.
  Rearranging this, it follows that $f_n$ and $f_n'$ are conjugate.
  Furthermore, since the Poisson deformation functor of $X$ is pro-representable, one can verify that we can choose $\theta_n$ to be an extension of $\theta_{n-1}$ (cf.~Proof of Prop.~\ref{prop:quat str of formal deform} (2)).
  Thus, the two morphisms $\{f_n\}_n$ and $\{f_n'\}_n$ on the formal Poisson deformation are conjugate.
\end{proof}

\begin{remark}\label{remark:uniqueness of f_n^m, C^*-equiv ver}
    In the assumption of Lemma \ref{lemma:uniqueness of f_n^m}, assume further that the formal Poisson deformation admits a $\mathbb{C}^*$-action and that the two automorphisms $\{f_n\}_n$ and $\{f_n'\}_n$ are equivariant with respect to this $\mathbb{C}^*$-action.
    By replacing the vector space $A_n$ used in the proof of Lemma \ref{lemma:uniqueness of f_n^m} with its $\mathbb{C}^*$-invariant subspace $A_n^{\mathbb{C}^*}$, a similar argument shows that $\{f_n\}_n$ and $\{f_n'\}_n$ are $\mathbb{C}^*$-equivariantly conjugate.
\end{remark}

With these preparations, we show the following proposition.

\begin{proposition}\label{prop:quat str of formal deform}
    Let $X$ be a conical symplectic variety with a good triple $(\lambda,\omega,j)$.
Then, the quaternionic structure $j$ on $X$ uniquely extends to a quaternionic structure $\{j_n\}_n$ on the formal universal Poisson deformation $\{(X_n,\omega_n)\}_n$ of $X$.
\end{proposition}

\begin{proof}
    We represent the formal universal Poisson deformation $\{(X_n,\omega_n)\}_n$ of $X$ as follows:
    \begin{equation}\label{diagram:formal deform of X}
       \vcenter{
       \xymatrix{
          X_0 \ar[d]\ar[r] &X_1 \ar[d] \ar[r]^{} & \cdots \ar[r]& X_n \ar[d] \ar[r]&\cdots \\
          S_0 \ar[r]&S_1 \ar[r]^{} & \cdots \ar[r]&  S_n \ar[r] &\cdots \\
       }}
    \end{equation}
    where $(X_0,\omega_0)=(X,\omega)$.
    We prove the claim by induction on the subscript $n$. For $n-1$, assume that a quaternionic structure $j_{n-1}$ on $X_{n-1}$ has been obtained.
    We show the following four claims in order.

    \begin{enumerate}[wide=\parindent, itemsep=\medskipamount]
       \item \textit{Extension to the base space}: By the universality of $(X_n,\omega_n)$, we obtain morphisms $j_n'$ and $\sigma_n$ making the following diagram commutative:
       \begin{equation}\label{diagram:formal j_n}
          \vcenter{
          \xymatrix{
                X \ar@{^{(}->}[r] \ar[d]_{j} &X_n \ar[d]_{\exists  j'_n} \ar[r]^{} \ar@{}[dr]
                & S_n \ar[d]^{\exists ! \sigma_n} \\
                \overline{X} \ar@{^{(}->}[r]&\overline{X_n} \ar[r]^{} & \overline{S_n} \\
          }}
       \end{equation}
       Since the morphism $\sigma_n$ is uniquely determined, using the universality again and the fact that $(j^2)^*\omega = \omega$, we can show that $\sigma_n^2=\mathrm{id}_{S_n}$.
       That is, we obtain a real structure $\{\sigma_n\}$ on the base space $\{S_n\}$.
       Note that since the morphism $j_n'$ on the total space is not unique, it does not necessarily define a quaternionic structure in general.

       \item \textit{Existence of induced Poisson morphisms $\{j_n\}$ on the total space}:
       We show that by rechoosing the morphism $j_n'$, we can obtain a Poisson morphism $\{j_n\}$ on $\{X_n\}$.
       Let $j'_n$ be the morphism obtained from the diagram (\ref{diagram:formal j_n}). Then,
       $j_{n-1}\circ (j_n'|_{X_{n-1}})^{-1}$
       is a Poisson automorphism on $X_{n-1}$. Moreover, restricting this morphism to $X$ gives the identity. Therefore, since the Poisson deformation functor $\mathrm{PD}_{(X,\omega)}$ of $(X,\omega)$ is pro-representable, this morphism extends to a Poisson morphism $\psi$ on $X_n$.
       Thus, if we set
       $j_n:=\psi \circ j_n'$,
       then $j_n$ is a Poisson morphism extending $j_{n-1}$.

\item \textit{Existence of quaternionic structure $\{j_n\}_{n}$}:
      To simplify the notation, we denote the Poisson deformation $(X_n,\omega_n)$ by $\eta_n$.
      Let
      \[
         A_n \coloneqq \mathrm{PAut}(\eta_n; \mathrm{id}|_{\eta_{n-1}})
      \]
      be the set of Poisson automorphisms of $\eta_n$ whose restrictions to $\eta_{n-1}$ are the identity.
      We denote the Poisson morphism induced by $j_n$ as $j_n:\eta_n\rightarrow \bar{\eta}_n$.
      By composing $j_n$ with the complex conjugation $\overline{(\cdot)}:\bar{\eta_n}\rightarrow \eta_n$, we obtain a Poisson morphism $\hat{j}_n :\eta_n\rightarrow \eta_n$.
      We aim to modify $j_n$ so that $\hat{j}_n \circ \hat{j}_n = \lambda_{n,-1}$, where $\lambda_{n,-1}$ denotes the action of $-1 \in \mathbb{C}^*$ on $\eta_n$.

      Let $G \coloneqq \mathbb{Z}/4\mathbb{Z} = \langle i : i^2 = -1 \rangle$.
      We define an action of $G$ on the vector space $A_n$ by
      \[
         {}^{i} u \coloneqq \hat{j}_n\circ u \circ \hat{j}_n^{-1} \quad \text{for } u \in A_n.
      \]
      Indeed, since $\hat{j}_{n-1}^2=\lambda_{n-1,-1}$, there exists an element $v\in A_n$ such that $\hat{j}_n^2=\lambda_{n,-1}\circ v$. Since $A_n$ is abelian, we have
      \[
         {}^{i}({}^{i} u)=\hat{j}_n^2\circ u \circ \hat{j}_n^{-2} = (\lambda_{n,-1}\circ v) \circ u \circ (\lambda_{n,-1}\circ v)^{-1} = {}^{-1}u.
      \]

      Next, we define a $2$-cocycle. Let $1_n \coloneqq \mathrm{id}_{\eta_n}$, $i_n \coloneqq \hat{j}_n$, and $(-1)_n \coloneqq \lambda_{n,-1}$.
      For any $\sigma, \tau \in G$, we define $f(\sigma,\tau) \in A_n$ by
      \[
         f(\sigma,\tau) \coloneqq \sigma_n\circ\tau_n\circ(\sigma \tau)^{-1}_n.
      \]
      Then, $f$ satisfies the cocycle condition:
      \[
         f(\sigma\tau,\rho)\circ f(\sigma,\tau\rho)^{-1}\circ f(\sigma,\tau) = {}^{\sigma}f(\tau,\rho),
      \]
      where ${}^{\sigma}f(\tau,\rho) \coloneqq \sigma_n\circ f(\tau,\rho)\circ \sigma_n^{-1}$.
      Therefore, $f$ defines a class in $H^{2}(G, A_n)$.
      Since $H^{2}(G, A_n)= 0$ by Lemma \ref{lemma:vanishing of group cohomology},
      $f$ is a $2$-coboundary.
      Thus, there exists a $1$-cochain $\{u_\sigma\}_{\sigma\in G}$ satisfying
      \[
         f(\sigma,\tau) = {}^\sigma u_\tau \circ u_{\sigma\tau}^{-1}\circ u_{\sigma}.
      \]
      Moreover, since $f(-1,1)=\mathrm{id}$, it follows that $u_{-1}=\mathrm{id}$.
      Then, we have
        \[
        \hat{j}_n^2\circ \lambda_{n,-1} = f(i,i) = {}^i u_i\circ \mathrm{id} \circ u_i = \hat{j}_n\circ u_i\circ \hat{j}_n^{-1}\circ u_i.
        \]
        Therefore, if we define
        \[
        \tilde{j}_n \coloneqq j_n\circ u_i^{-1},
        \]
        it satisfies $\overline{\tilde{j}_n}\circ \tilde{j}_n = \lambda_{n,-1}$ and defines a quaternionic structure on $X_n$ extending the Poisson morphism $j_{n-1}$.

       \item \textit{Uniqueness of quaternionic structure}:
       Given two quaternionic structures $j_n$ and $j_n'$, let $f_n$ and $f_n'$ be their compositions with the complex conjugation $\overline{(\cdot)}:\bar{\eta_n}\rightarrow \eta_n$, respectively.
       By Lemma \ref{lemma:uniqueness of f_n^m}, $f_n$ and $f_n'$ are conjugate, which implies that $j_n$ and $j_n'$ are also conjugate.\qedhere
    \end{enumerate}
\end{proof}

\begin{remark} \label{rem:good triple on formal PD}
    By using the uniqueness of the $\mathbb{C}^*$-action $\{\lambda_n\}_n$ on the formal universal Poisson deformation $\{(X_n,\omega_n)\}_n$, we can rechoose the quaternionic structure $\{j_n\}_n$ so that it satisfies the compatibility conditions of a good triple on a Poisson deformation in Definition~\ref{def:good triple on PD}. We call the triple $\{(\lambda_n,\omega_n,j_n)\}_n$ thus obtained a \textit{good triple} on the formal universal Poisson deformation.
\end{remark}

Next, we show that by algebraizing the good triple on the formal universal Poisson deformation, we obtain the good triple on the universal Poisson deformation.

\begin{corollary}\label{cor:algebraization of quat str}
    Let $X$ be a conical symplectic variety with a good triple $(\lambda,\omega,j)$.
    Then, the good triple on $X$ naturally induces a unique good triple $(\tilde{\lambda}, \tilde{\omega}, \tilde{j})$ on the universal Poisson deformation $\mathcal{X}\rightarrow\mathcal{B}$ of $X$.
\end{corollary}

\begin{proof}
    By Proposition~\ref{prop:quat str of formal deform} and Remark~\ref{rem:good triple on formal PD}, the good triple on $X$ induces a good triple $\{(\lambda_n,\omega_n,j_n)\}_n$ on the formal universal Poisson deformation $\{X_n\}_n$ of $X$. Let $P_n$ be the coordinate ring of $X_n$, and let
    $P=\varprojlim P_n$
    be the projective limit with respect to $n$.
    The good triple $\{(\lambda_n,\omega_n,j_n)\}_n$ on $\{X_n\}_n$ induces the triple $(\tilde{\lambda}, \tilde{\omega}, \tilde{j})$ on $P$.
    According to Namikawa's proof of Theorem~\ref{Namikawa Thm: C^*-action on Yuni} (cf.~\cite[Prop.~A.7]{Namikawa08}), the coordinate ring $B$ of the universal Poisson deformation $\mathcal{X}$ is characterized as the finitely generated $\mathbb{C}$-algebra generated by the eigenvectors of the $\mathbb{C}^*$-action on the ring $P$. Here, an eigenvector $v$ of the $\mathbb{C}^*$-action on the ring $P$ is an element $v$ of $P$ such that there exists an integer $w$ satisfying $t\cdot v = t^w v$ for any $t\in\mathbb{C}^*$.
    Since the quaternionic structure $\tilde{j}$ on $P$ is compatible with the $\mathbb{C}^*$-action $\tilde{\lambda}$, it preserves the subalgebra $B\subset P$. Therefore, by restricting the triple $(\tilde{\lambda}, \tilde{\omega}, \tilde{j})$ on $B$, we obtain the good triple on the universal Poisson deformation $\mathcal{X}$.
\end{proof}

With these preparations, we prove Proposition~\ref{prop: good triple on Yuni}.

\begin{proof}[Proof of Proposition~\ref{prop: good triple on Yuni}]
\textit{Restatement of the claim:
    Let $X$ be a conical symplectic variety with a good triple $(\lambda,\omega,j)$. Assume that $X$ admits a crepant resolution $Y$.
    Let $\mathcal{Y}\rightarrow\mathcal{C}$ be the universal Poisson deformation of the crepant resolution $Y$.
    Then, the affinization $\mathcal{X}'\rightarrow\mathcal{C}$ of $\mathcal{Y}$ admits a good triple $(\tilde{\lambda}', \tilde{\omega}', \tilde{j}')$ naturally induced from the good triple on $X$.
    Moreover, this good triple naturally lifts to a unique triple $(\tilde{\lambda}_Y, \tilde{\omega}_Y, \tilde{j}_Y)$ on $\mathcal{Y}$.}

    \medskip

    We give the proof below. By Corollary~\ref{cor:algebraization of quat str}, we obtain the good triple $(\tilde{\lambda}, \tilde{\omega}, \tilde{j})$ on the universal Poisson deformation $\mathcal{X}\rightarrow\mathcal{B}$ of $X$.
    The deformation $\mathcal{X}'\rightarrow\mathcal{C}$ is, by definition, the Poisson deformation obtained as the pullback of $\mathcal{X}\rightarrow\mathcal{B}$ via the quotient map $q:\mathcal{C}\rightarrow\mathcal{B}$ (cf.~Theorem~\ref{Namikawa Thm:CM and univ Poisson}).
    As seen in the proof (1) of Proposition~\ref{prop:quat str of formal deform}, the quaternionic structure $j$ induces a real structure on the base space $\mathcal{B}$. Thus, one can choose a real structure $\sigma_\mathcal{C}$ on $\mathcal{C}$ that is compatible with the quotient map $q$. With this choice, using the universality of the pullback of the deformation $\mathcal{X}'$, we obtain the good triple $(\tilde{\lambda}', \tilde{\omega}', \tilde{j}')$ on $\mathcal{X}'$.
    \medskip

    Next, we show that the good triple $(\tilde{\lambda}', \tilde{\omega}', \tilde{j}')$ on the deformation $\mathcal{X}'$ lifts to the simultaneous resolution $\mathcal{Y}\rightarrow\mathcal{C}$. Let $\nu:\mathcal{Y}\rightarrow \mathcal{X}'$ be the affinization map.
    Let $U \subset \mathcal{X}'$ be the regular locus of the resolution $\nu$.
    Since the exceptional set $\mathcal{E} \coloneqq \mathcal{Y} \setminus \nu^{-1}(U)$ has codimension at least $2$ in $\mathcal{Y}$, the good triple on $U$ naturally lifts to a triple on $\nu^{-1}(U)$, which induces a birational anti-holomorphic involution $\tilde{j}_{\mathcal{Y}}$ on $\mathcal{Y}$ defined outside $\mathcal{E}$.

    To show that this birational map extends to a regular automorphism on the whole $\mathcal{Y}$, we take a relative $\nu$-ample line bundle $\mathcal{L}$ on $\mathcal{Y}$. As in \cite{namikawa2026notes}~\cite[\S 7]{Namikawa08}, such $\mathcal{L}$ can be obtained by algebraizing a compatible line bundle $L_n$ on the formal deformation $\{Y_n\}$ corresponding to $\mathcal{Y}$.
    By replacing $\mathcal{L}$ with $\mathcal{L} \otimes \overline{(\tilde{j}_{\mathcal{Y}})^*\mathcal{L}}$ if necessary, we can choose $\mathcal{L}$ to be compatible with the quaternionic structure, satisfying $(\tilde{j}_{\mathcal{Y}})^*\mathcal{L} \simeq \overline{\mathcal{L}}$ on $\nu^{-1}(U)$.
    Taking the pushforward by $\nu$, we obtain an isomorphism of reflexive sheaves $(\tilde{j}')^*(\nu_*\mathcal{L}^{\otimes m}) \simeq \overline{\nu_*\mathcal{L}^{\otimes m}}$ on $U$ for any $m \ge 0$.
    Since $\mathcal{X}'$ is a normal variety and the singular locus $\mathcal{X}' \setminus U$ has codimension at least $2$ in $\mathcal{X}'$, this isomorphism uniquely extends to the whole $\mathcal{X}'$ by Hartogs' extension theorem.
    Because $\mathcal{L}$ is relative $\nu$-ample, we have an isomorphism $\mathcal{Y} \simeq \operatorname{Proj}_{\mathcal{X}'} \bigoplus_{m \ge 0} \nu_* \mathcal{L}^{\otimes m}$. Thus, the birational map $\tilde{j}_{\mathcal{Y}}$ uniquely extends to a regular automorphism on $\mathcal{Y}$, and the triple lifts to a good triple on $\mathcal{Y}$.
\end{proof}

\begin{remark} \label{rem:C^* and real str on base sp}
    The triple on the deformation $\mathcal{Y}$ naturally induces scalar multiplication of weight 2 and a real structure on the base space $\mathcal{C}$ (cf.~Theorem~\ref{Namikawa Thm: C^*-action on Yuni} and the proof (1) of Prop.~\ref{prop:quat str of formal deform}).
\end{remark}

\begin{remark} \label{rem:triple on Y}
    By restricting the triple $(\tilde{\lambda}_Y, \tilde{\omega}_Y, \tilde{j}_Y)$ on $\mathcal{Y}$ to the central fiber $Y$, we obtain a triple $(\lambda_Y, \omega_Y, j_Y)$.
    Consequently, this implies that the good triple on $X$ naturally lifts to a triple on the crepant resolution $Y$.
\end{remark}

\section{Principal twistor model}\label{sec:PTM}

In this section, we first give the definitions of a twistor model, which serves as a candidate for a twistor space, and a principal twistor model obtained by generalizing it. Next, we show that when a conical symplectic variety with a good triple admits a crepant resolution, we can construct the principal twistor model $\mathcal{Y}(1)$ by using the $\mathbb{C}^*$-action to glue the universal Poisson deformation of the resolution.
Finally, we show the universality of the principal twistor model, stating that any algebraic twistor space corresponding to a hyperkähler metric with asymptotic behavior can be obtained by slicing the principal twistor model.

\subsection{Definitions}\label{sec:def of PTM}
After recalling the definition of a twistor space, we introduce the notions of a twistor model as a candidate for it, and a principal twistor model as a generalization of the latter.
We then explain that, given a principal twistor model, one can obtain a family of twistor models by slicing it along arbitrary real sections.

\medskip
First, we recall the definition of the twistor space of a hyperkähler manifold.

\begin{definition}[Twistor space]\label{def:twistor space}
    Let $n$ be a natural number. Consider the following four conditions for a $(2n+1)$-dimensional complex manifold $Z$:
    \begin{enumerate}
       \item[(T1)] $Z$ admits a holomorphic fibration structure $\pi:Z\rightarrow \mathbb{P}^1$.
       Here, a \textit{fibration} means a surjective submersion.

       \item[(T2)] $Z$ admits a real structure $\tau$ compatible with the antipodal map $\sigma_{\mathrm{ap}}(u)\coloneqq-\dfrac{1}{\bar{u}}$ on $\mathbb{P}^1$. That is, it satisfies $\pi\circ \tau = \sigma_{\mathrm{ap}} \circ \pi$.

       \item[(T3)] There exists a relative holomorphic symplectic $2$-form $\omega$ on $\pi$ twisted by $\pi^*\mathcal{O}(2)$. That is, $\omega \in \Gamma(Z, \Omega^2_{\pi}(2)).$
           Furthermore, $\omega$ is compatible with the real structure $\tau$:
           \[\tau^*\omega=\overline{\omega}.\]

       \item[(T4)] There exists a family of twistor lines $\{\ell_x\}_{x\in M}$ parameterized by a real $4n$-dimensional manifold $M$, and this family forms a foliation of $Z$.
           Here, a \textit{twistor line} is a holomorphic section $\ell$ of the fibration $\pi$ satisfying the following two conditions:
       \begin{enumerate}
           \item The section $\ell$ is a real section of $\pi$. That is, it satisfies $\tau\circ \ell = \ell \circ \sigma_{\mathrm{ap}}$.
           \item The normal bundle of $\ell$ in $Z$ is isomorphic to $N_{\ell/Z}\simeq \mathbb{C}^{2n}\otimes \mathcal{O}(1)$.
       \end{enumerate}
    \end{enumerate}
    \medskip
    When a $(2n+1)$-dimensional complex manifold $Z$ satisfies the above four conditions, we call the tuple $(Z,\omega,\tau,\{\ell_x\}_{x \in M})$ a \textit{twistor space}.
\end{definition}

\begin{remark}
    We sometimes abbreviate the tuple and simply call the complex manifold $Z$ a twistor space.
\end{remark}

The fundamental theorem on twistor spaces is as follows.
\begin{theorem} [Twistor correspondence \cite{HKLR}] \label{thm:twistor correspondence}
    Twistor spaces $(Z,\omega,\tau,\{\ell_x\}_{x \in M})$ are in natural one-to-one correspondence with hyperkähler structures $(M,g,I,J,K)$ on the parameter space $M$ of the family of twistor lines, up to isomorphism.
\end{theorem}

\begin{remark}
    To simplify the description, we refer to the twistor space $Z$ corresponding to a hyperkähler structure $(M,g,I,J,K)$ simply as the twistor space $Z$ corresponding to the hyperkähler metric $g$, omitting the complex structures.
\end{remark}

\begin{remark}
    In the definition of a twistor space, we do not assume the positive definiteness of the induced metric. If necessary, one can switch the signature of the metric by changing the sign of the holomorphic symplectic form $\omega$.
\end{remark}

Next, we introduce the definition of a twistor model.
We define a twistor model to be a tuple obtained by removing the condition on twistor lines from the conditions for a twistor space:

\begin{definition}[Twistor model]\label{def:twistor model}
    Let $Z$ be a $(2n+1)$-dimensional complex manifold. If $Z$ satisfies conditions (T1), (T2), and (T3) in Definition~\ref{def:twistor space}, we call the tuple $(Z,\omega,\tau)$ a \textit{twistor model}.
\end{definition}

We introduce the definition of a principal twistor model obtained by generalizing the definition of a twistor model.

\begin{definition}[Principal twistor model]\label{def:PTM}
    Let $n$ be a natural number and $d$ a non-negative integer. Consider the following three conditions for a $(2n+d+1)$-dimensional complex manifold $P$:
    \begin{enumerate}
       \item[(P1)] $P$ admits a holomorphic fibration structure $\pi_P:P\rightarrow \mathbb{P}^1$. Moreover, for a complex $d$-dimensional vector space $\mathcal{C}$, there exists a holomorphic bundle map $\varphi$ to the total space of the vector bundle $\pi_{\mathcal{C}}:\mathcal{C}(2)=\mathcal{C}\otimes \mathcal{O}(2)\rightarrow \mathbb{P}^1$ making the following diagram commute:
       \begin{equation*}
           \xymatrix{
           P\ar[rr]^-{\varphi}\ar[dr]_-{\pi_P}&&\mathcal{C}(2)\ar[dl]^-{\pi_{\mathcal{C}}}\\
           &\mathbb{P}^1&
           }
       \end{equation*}

       \item[(P2)] $P$ admits a real structure $\tau_P$ compatible with the real structure $\sigma_2$ on the vector bundle $\mathcal{C}(2)$ (cf.~Example~\ref{ex:quat str on O(k)}). That is, the following diagram commutes:
       \begin{equation*}
           \xymatrix{
           P\ar[r]^-{\tau_P}\ar[d]_-{\varphi}&\overline{P}\ar[d]^-{\bar{\varphi}}\\
           \mathcal{C}(2)\ar[r]_-{\sigma_2}&\overline{\mathcal{C}(2)}
           }
       \end{equation*}

       \item[(P3)] There exists a relative holomorphic symplectic $2$-form $\omega_P$ on $\varphi$ twisted by $\pi_P^*\mathcal{O}(2)$ on $P$. That is, $\omega_P \in \Gamma(P, \Omega^2_{\varphi}(2)).$
           Furthermore, $\omega_P$ is compatible with the real structure $\tau_P$:
           \[\tau_P^*\omega_P=\overline{\omega_P}.\]

    \end{enumerate}

    When a $(2n+d+1)$-dimensional complex manifold $P$ satisfies the above three conditions, we call the tuple $(\varphi:P\rightarrow \mathcal{C}(2),\omega_P,\tau_P)$ a \textit{principal twistor model}.
\end{definition}

\begin{remark}
    We sometimes abbreviate the tuple and simply call the complex manifold $P$ a principal twistor model.
    Note that when $d=0$, the definition of a principal twistor model coincides with that of a twistor model.
\end{remark}

Next, we define a real section of the vector bundle $\mathcal{C}(2)$:

\begin{definition}\label{def:real section of C(2)}
    Consider the natural real structure $\sigma_2$ on the vector bundle $\mathcal{C}(2)\rightarrow\mathbb{P}^1$ (cf.~Example~\ref{ex:quat str on O(k)}). Let $\sigma_{\mathrm{ap}}$ be the antipodal map on $\mathbb{P}^1$.
    We say that a holomorphic section $s$ of the vector bundle $\mathcal{C}(2)$ is a \textit{real section} if the section $s$ is compatible with the real structures, that is,
    \[\sigma_2 \circ s = s \circ \sigma_{\mathrm{ap}}.\]
\end{definition}

The following is a fundamental proposition on principal twistor models. We state that, for any choice of a real section $s$ of the vector bundle $\mathcal{C}(2)$, we obtain a twistor model $Z_s$ by slicing the principal twistor model along the section $s$.
The proof of the following proposition is straightforward and thus omitted.
\begin{proposition}\label{prop:slicing of PT}
Let $(\varphi:P\rightarrow \mathcal{C}(2),\omega_P,\tau_P)$ be a principal twistor model.
For any real section $s\in H^0(\mathbb{P}^1,\mathcal{C}(2))^{\sigma_2}$ of $\pi_\mathcal{C}:\mathcal{C}(2)\rightarrow \mathbb{P}^1$, if we define
\begin{align*}
    \pi_s&\coloneqq s^*\varphi : Z_s\rightarrow \mathbb{P}^1,\\
    \omega_s&\coloneqq\omega_P|_{Z_s},\\
    \tau_s&\coloneqq\tau_P|_{Z_s},
\end{align*}
then the tuple $(Z_s,\omega_s,\tau_s)$ is a twistor model.
\end{proposition}

\begin{remark}
    The idea of constructing a twistor space by slicing a larger manifold, such as a principal twistor model, along a real section is also used in Santa-Cruz ('97) \cite{Santa-Cruz97} and Bielawski-Foscolo ('21) \cite{Bielawski-Foscolo21}.
\end{remark}

\subsection{$\mathbb{C}^*$-gluing construction} \label{subsec:C^*-gluing construction}
In this subsection, we explain the $\mathbb{C}^*$-gluing construction, which naturally constructs a holomorphic fiber bundle $M(1)$ over $\mathbb{P}^1$ from a complex manifold $M$ with a $\mathbb{C}^*$-action. We also show that the $\mathbb{C}^*$-action, quaternionic structure, and differential forms on $M$ naturally extend to the fiber bundle $M(1)$ in order.
Finally, we show that the principal twistor model $\mathcal{Y}(1)$ can be constructed by applying the $\mathbb{C}^*$-gluing construction to the triple on the universal Poisson deformation.

\subsubsection{Definition} \label{subsubsec:def of C^*-gluing}

Let $M$ be a complex manifold equipped with a $\mathbb{C}^*$-action. We define a complex manifold $M(1)$ obtained by gluing two copies of $M\times\mathbb{C}$ via the following gluing function $\psi$:
\begin{align}
    M\times\mathbb{C}^* &\overset{\psi}{\longrightarrow}M\times\mathbb{C}^*\notag\\
    (x,u)&\longmapsto \left(\dfrac{1}{u}\cdot x,\dfrac{1}{u}\right)\eqqcolon(x',v). \label{gluing fn}
\end{align}
By the natural projection to the second factor, $M(1)$ admits a holomorphic fiber bundle structure $M(1)\rightarrow \mathbb{P}^1$.
We call the above construction the \textit{$\mathbb{C}^*$-gluing construction}.

\begin{remark}
    The notation $M(1)$ implies that the gluing has weight $1$ with respect to the given $\mathbb{C}^*$-action.
    Note also that the same construction applies equally well even if $M$ is a complex variety with singularities.
\end{remark}

\begin{example}
    When $M$ is a vector space $V$, if we give a $\mathbb{C}^*$-action on $V$ as scalar multiplication of weight $k\in \mathbb{Z}$, then we have
    \[M(1)=V(k)\coloneqq V\otimes \mathcal{O}(k).\]
\end{example}

\subsubsection{Extension of $\mathbb{C}^*$-action}

We state that the $\mathbb{C}^*$-action on the complex manifold $M$ naturally induces two types of $\mathbb{C}^*$-actions on the fiber bundle $M(1)$. The proof of the following proposition is straightforward and thus omitted.
\begin{proposition}\label{prop:C^*- on M(1)}
    Let $M$ be an arbitrary complex manifold with a $\mathbb{C}^*$-action, and let $M(1)$ be the fiber bundle constructed by the $\mathbb{C}^*$-gluing construction. Let $\psi$ be the gluing function \eqref{gluing fn} for the fiber bundle $M(1)$, and let $(x,u)$ and $(x',v)$ denote the points on the two copies of $M\times\mathbb{C}$ associated with it. Then, the following two types of $\mathbb{C}^*$-actions are naturally induced on $M(1)$:
    \begin{enumerate}[leftmargin=5em]
       \item[\textit{Type 1} : ] For any $\lambda \in\mathbb{C}^*$,
          \begin{equation}
             \left\{
             \begin{aligned}
                \lambda \cdot (x,u) &= (\lambda \cdot x, \lambda^2u),\\
                \lambda \cdot (x',v)&= (\lambda^{-1} \cdot x', \lambda^{-2}v).
             \end{aligned} \label{eq:type 1}
             \right.
          \end{equation}

       \item[\textit{Type 2} : ] For any $\lambda \in\mathbb{C}^*$,
          \begin{equation}
             \left\{
             \begin{aligned}
                \lambda \cdot (x,u) &= (\lambda \cdot x, u),\\
                \lambda \cdot (x',v)&= (\lambda \cdot x', v).
             \end{aligned} \label{eq:type 2}
             \right.
          \end{equation}
    \end{enumerate}
\end{proposition}

Next, we state that a $\mathbb{C}^*$-equivariant morphism between two complex manifolds $M$ and $N$ naturally extends to a $\mathbb{C}^*$-equivariant bundle map on the fiber bundles. The proof of the following proposition is straightforward and thus omitted.

\begin{proposition}\label{prop: C^* equiv map on M(1)}
    Let $M$ and $N$ be arbitrary complex manifolds with $\mathbb{C}^*$-actions, and let $M(1)$ and $N(1)$ be the fiber bundles constructed by the $\mathbb{C}^*$-gluing construction. Given a $\mathbb{C}^*$-equivariant map $\varphi:M\rightarrow N$ between $M$ and $N$, a bundle map $\varphi(1):M(1)\rightarrow N(1)$ between the fiber bundles $M(1)$ and $N(1)$ is naturally defined by $(x,u)\mapsto (\varphi(x),u)$. Then, the bundle map $\varphi(1)$ is equivariant with respect to the $\mathbb{C}^*$-action of \textit{Type 1} \textup{(}resp. \textit{Type 2}\textup{)} on the fiber bundles $M(1)$ and $N(1)$.
\end{proposition}

\subsubsection{Extension of quaternionic structure}

We show that when a quaternionic structure is given on a complex manifold $M$ with a $\mathbb{C}^*$-action, a natural real structure is induced on the fiber bundle $M(1)$. First, we show the following proposition on anti-holomorphic maps.

\begin{proposition}\label{prop:anti-hol map on M(1)}
    Let $M$ and $N$ be two complex manifolds with $\mathbb{C}^*$-actions. Given an anti-holomorphic map $f:M\rightarrow \bar{N}$ that is compatible with the $\mathbb{C}^*$-actions in the following sense:
    For any $t \in \mathbb{C}^*$,
    \[
    f(t\cdot x)=\bar{t}\cdot f(x).
    \]
    Then, a natural anti-holomorphic map $\hat{f}:M(1)\rightarrow \bar{N}(1)$ on $M(1)$ is determined as follows:
    \begin{equation}
       \left\{
       \begin{aligned}
          \hat{f}(x,u)&=\left(-\dfrac{1}{\bar{u}}\cdot f(x),-\dfrac{1}{\bar{u}}\right),\\
          \hat{f}(x',v)&=\left(-\dfrac{1}{\bar{v}}\cdot -f(x'),-\dfrac{1}{\bar{v}}\right),
       \end{aligned} \label{defeq:hat f}
       \right.
    \end{equation}
    where $-f$ denotes the composition with the action of $-1$, i.e., $-f\coloneqq-1\cdot f(x)$.

\end{proposition}

\begin{proof}
  Let $\psi_M$ and $\psi_N$ be the gluing functions of $M(1)$ and $N(1)$ given in the definition of the $\mathbb{C}^*$-gluing construction (\S\ref{subsubsec:def of C^*-gluing}), respectively. In terms of coordinates $(x,u)$ on $M\times\mathbb{C}^*$,
  a simple calculation yields the following identity:
\[\hat{f}\circ \psi_M (x,u)=(f(x),-\bar{u})=\psi_N\circ \hat{f}(x,u).\]
Therefore, the anti-holomorphic map $\hat{f}$ is well-defined.
\end{proof}

\begin{remark}
    By construction, $\hat{f}$ is compatible with the antipodal map $\sigma_{\mathrm{ap}}$ on $\mathbb{P}^1$.
\end{remark}

Next, we show the following proposition on the extension of quaternionic structures and real structures.

\begin{proposition}\label{prop:real str on M(1)}
    Let $M$ be a complex manifold equipped with a $\mathbb{C}^*$-action.
    Suppose that an anti-holomorphic automorphism $f$ on $M$ compatible with the $\mathbb{C}^*$-action (cf.~Prop.~\ref{prop:anti-hol map on M(1)}) is given.
    Then, for the anti-holomorphic map $\hat{f}$ on the fiber bundle $M(1)$ given by Proposition~\ref{prop:anti-hol map on M(1)}, the following hold:

    \begin{enumerate}
       \item The map $\hat{f}$ is a real structure on the fiber bundle $M(1)$ if and only if the map $f$ is a quaternionic structure on $M$ (cf.~Def.~\ref{def:quat str}).
       \item The map $\hat{f}$ is a quaternionic structure on the fiber bundle $M(1)$ with respect to the (fiberwise) $\mathbb{C}^*$-action of \textit{Type 2} \eqref{eq:type 2} if and only if the map $f$ is a real structure on $M$.
    \end{enumerate}

\end{proposition}

\begin{proof}
    As in \S\ref{subsubsec:def of C^*-gluing}, in terms of coordinates $(x,u)$ on $M\times\mathbb{C}^*$, a simple calculation yields the following identity:
    \[\hat{f}^2(x,u) = (-f^2(x),u).\]
    The assertion follows immediately from this.
\end{proof}

We introduce specific examples of extensions of quaternionic structures and real structures. The following examples appear naturally in the discussion of twistor spaces.

\begin{example}\label{ex:quat str on O(k)}
    Let $M=V$ be a complex vector space, and consider a $\mathbb{C}^*$-action on it by scalar multiplication of weight $k\in\mathbb{Z}$. Then, we obtain the vector bundle $M(1)=V(k)$.
    Consider the real structure $f(z) = \bar{z}$ on $V$ given by complex conjugation.
    Then, the real structure $f$ induces a quaternionic structure $\sigma_k\coloneqq\hat{f}$ on $V(k)$:
\[\sigma_k(z,u)=\left((-1)^k\dfrac{1}{\bar{u}^k}\bar{z},-\dfrac{1}{\bar{u}}\right).\]
\end{example}

\begin{remark}
When $V=\mathbb{C}$, this quaternionic structure $\sigma_k$ coincides with the anti-holomorphic map on $\mathcal{O}(k)$ considered by Hitchin \cite{Hitchin79}.
\end{remark}

\begin{example} \label{ex:real str on O(1)+O(1)}
    Let $V=\mathbb{C}^2$ and consider a $\mathbb{C}^*$-action by scalar multiplication.
    Then, we obtain the vector bundle $V(1)=\mathcal{O}(1)\oplus \mathcal{O}(1)$.
    If we consider the quaternionic structure $j(z_1,z_2)\coloneqq(\bar{z_2},-\bar{z_1})$ on $V$, then we obtain a real structure $\tau\coloneqq\hat{j}$ on the vector bundle $V(1)$:
    \begin{equation*}
       \tau(z_1,z_2,u) = \left(-\dfrac{1}{\bar{u}}\bar{z_2},\dfrac{1}{\bar{u}}\bar{z_1},-\dfrac{1}{\bar{u}}\right).
    \end{equation*}
\end{example}

\begin{remark}
    This real structure $\tau$ coincides with the real structure on the twistor space $\mathcal{O}(1)\oplus \mathcal{O}(1)$ corresponding to the flat metric on $V=\mathbb{C}^2$.
\end{remark}

\subsubsection{Extension of holomorphic differential forms}

We show that when a holomorphic differential form on a manifold $M$ with a $\mathbb{C}^*$-action is compatible with the $\mathbb{C}^*$-action, a relative holomorphic differential form is naturally induced on the fiber bundle $M(1)$. The proof of the following proposition is straightforward and thus omitted.

\begin{proposition} \label{prop:diff form on M(1)}
    Let $M$ be a manifold with a $\mathbb{C}^*$-action. Suppose that a holomorphic $p$-form $\omega$ satisfying $\lambda^*\omega=\lambda^k\omega\ (k\in \mathbb{Z})$ is given.
    Let $pr_1$ be the projection to the first factor of $M\times \mathbb{C}^*$. By gluing the differential form $pr_1^*\omega$ on $M\times \mathbb{C}^*$ via the $\mathbb{C}^*$-action, we obtain a relative holomorphic differential form $\hat{\omega}$ on $M(1)$ twisted by $\mathcal{O}(k)$:
\[\hat{\omega}\in \Gamma(M(1),\Omega^p_{M(1)/\mathbb{P}^1}\varphi(k)).\]
\end{proposition}

\subsubsection{Construction of the principal twistor model}

Let $X$ be a conical symplectic variety with a good triple $(\lambda,\omega,j)$.
Assume that $X$ admits a crepant resolution $Y$.
In this subsection, we show that a principal twistor model $\mathcal{Y}(1)$ can be naturally constructed by applying the $\mathbb{C}^*$-gluing construction to the triple induced on the universal Poisson deformation $\mathcal{Y}$ of the crepant resolution $Y$.

\medskip
As seen in Proposition~\ref{prop: good triple on Yuni}, the good triple on $X$ naturally induces the triple $(\tilde{\lambda}_Y,\tilde{\omega}_Y,\tilde{j}_Y)$ on the total space of the universal Poisson deformation $\varphi_Y:\mathcal{Y}\rightarrow \mathcal{C}$ of the resolution $Y$.
In this case, scalar multiplication $\lambda_\mathcal{C}=\lambda_{\mathrm{Scal}}^2$ of weight $2$ and a real structure given by complex conjugation on $\mathcal{C}$ are induced on the base space $\mathcal{C}$, as seen in Remark~\ref{rem:C^* and real str on base sp}.
Thus, by applying the $\mathbb{C}^*$-gluing construction to these, we obtain the vector bundle $\mathcal{C}(2)$ and the real structure $\sigma_2$ on it (cf.~Example~\ref{ex:quat str on O(k)}).
Since the map $\varphi_Y:\mathcal{Y}\rightarrow \mathcal{C}$ is equivariant with respect to the $\mathbb{C}^*$-action, by Proposition~\ref{prop: C^* equiv map on M(1)}, we obtain a fiber bundle $P\coloneqq\mathcal{Y}(1)\rightarrow \mathbb{P}^1$ and a holomorphic bundle map
\[\varphi_Y(1):P\rightarrow \mathcal{C}(2).\]
This satisfies condition (P1) in the definition of a principal twistor model (Def.~\ref{def:PTM}).
Furthermore, applying Proposition~\ref{prop:real str on M(1)} to the quaternionic structure $\tilde{j}_Y$, we obtain a real structure $\tau_P\coloneqq\widehat{\tilde{j}}_Y$ on $P$.
By construction, the real structure $\tau_P$ on $P$ and the real structure $\sigma_2$ on $\mathcal{C}(2)$ are compatible with respect to the map $\varphi_Y(1)$.
That is, they satisfy condition (P2).
Finally, combined with Proposition~\ref{prop:diff form on M(1)}, we obtain a relative holomorphic symplectic $2$-form $\omega_P\coloneqq\widehat{\tilde{\omega}}_Y \in \Gamma(P, \Omega^2_{\varphi_Y(1)}(2))$.
By definition, the real structure $\tau_P$ and the relative $2$-form $\omega_P$ are compatible.
That is, they satisfy condition (P3).
Therefore, the tuple $(\varphi_Y(1):P\rightarrow \mathcal{C}(2), \omega_P, \tau_P)$ is a principal twistor model.
Consequently, we obtain the following proposition.

\begin{proposition}\label{main prop: PTM of Y}

Let $X$ be a conical symplectic variety with a good triple $(\lambda,\omega,j)$.
Assume that $X$ admits a crepant resolution $Y$.
The good triple on $X$ naturally induces the triple $(\tilde{\lambda}_Y,\tilde{\omega}_Y,\tilde{j}_Y)$ on the total space of the universal Poisson deformation $\varphi_Y:\mathcal{Y}\rightarrow \mathcal{C}$ of the resolution $Y$ by Proposition~\ref{prop: good triple on Yuni}.
By applying the $\mathbb{C}^*$-gluing construction to this triple, we obtain a fiber bundle $\varphi_Y(1):\mathcal{Y}(1)\rightarrow \mathcal{C}(2)$, a real structure $\tau_P$, and a relative holomorphic symplectic $2$-form $\omega_P$ on $\mathcal{Y}(1)$. In this case, the tuple $(\varphi_Y(1):\mathcal{Y}(1)\rightarrow \mathcal{C}(2),\omega_P,\tau_P)$ is a principal twistor model.

\end{proposition}

\begin{remark}\label{remark:sing PTM of X}
    The good triple on $X$ induces the good triple $(\tilde{\lambda}',\tilde{\omega}',\tilde{j}')$ on the affinization $\varphi':\mathcal{X}'\rightarrow\mathcal{C}$ of the universal Poisson deformation $\mathcal{Y}\rightarrow\mathcal{C}$ of the resolution $Y$ by Proposition~\ref{prop: good triple on Yuni}.
    By applying the same construction as in Theorem~\ref{main thm: universality of PTM} for this triple, we obtain the tuple $(\varphi'(1):\mathcal{X}'(1)\rightarrow \mathcal{C}(2), \omega_{P}', \tau_{P}')$. We call this tuple the \textit{singular principal twistor model} of $\mathcal{Y}(1)$.
\end{remark}

We define the following for the discussion in the subsequent sections.

\begin{definition}\label{def:twistor cone}
    Let $X$ be a conical symplectic variety with a good triple $(\lambda,\omega,j)$.
    We call the twistor model $(X(1)\rightarrow\mathbb{P}^1,\hat{\omega},\hat{j})$ obtained by using the $\mathbb{C}^*$-gluing construction for this good triple a \textit{twistor cone}.
\end{definition}

\begin{remark}
    The twistor cone $X(1)$ coincides with the singular twistor model obtained by slicing the singular principal twistor model $\mathcal{X}'(1)\rightarrow\mathcal{C}(2)$ along the zero section of $\mathcal{C}(2)$.
\end{remark}

\begin{remark}[Conceptual diagrams of the principal twistor model $\mathcal{Y}(1)$] \label{remark:Conceptual diagrams of PTM}
    \hfill \break
    Let $Y(1)$ be the twistor model obtained by slicing the principal twistor model $\mathcal{Y}(1)$ along the zero section of the vector bundle $\mathcal{C}(2)$.
    Note that $Y(1)$ is the natural simultaneous resolution of the twistor cone $X(1)$.
    Figure~\ref{fig:PTM} provides a conceptual diagram of the principal twistor model $\mathcal{Y}(1)$, represented as a rectangular box, where the three coordinate axes correspond to $\mathcal{C}$, $Y$, and $\mathbb{P}^1$. The three faces of this box correspond to the vector bundle $\mathcal{C}(2)$, the universal Poisson deformation family $\mathcal{Y}$, and the resolution $Y(1)$, respectively.
    Furthermore, Figure~\ref{fig:twistor model in PTM} is a conceptual diagram illustrating how a twistor model $Z_s$, obtained by slicing $\mathcal{Y}(1)$ along a real section $s$ of $\mathcal{C}(2)$, is embedded inside the principal twistor model $\mathcal{Y}(1)$.
\end{remark}

\begin{figure}[htbp]
\centering
\begin{minipage}[t]{.46\linewidth}%
    \centering
    \includegraphics[height = 4.8cm]{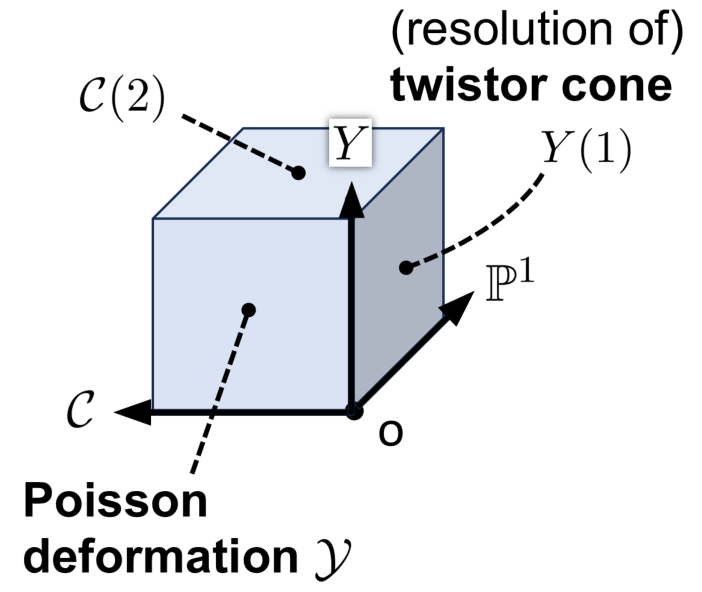}
    \caption{A conceptual diagram of the principal twistor model $\mathcal{Y}(1)$, where the three coordinate axes correspond to $\mathcal{C}$, $Y$, and $\mathbb{P}^1$. The three faces correspond to $\mathcal{C}(2)$, $\mathcal{Y}$, and $Y(1)$.}
    \label{fig:PTM}
\end{minipage}
\hfill
\begin{minipage}[t]{.46\linewidth}%
    \centering
    \includegraphics[height = 4.8cm]{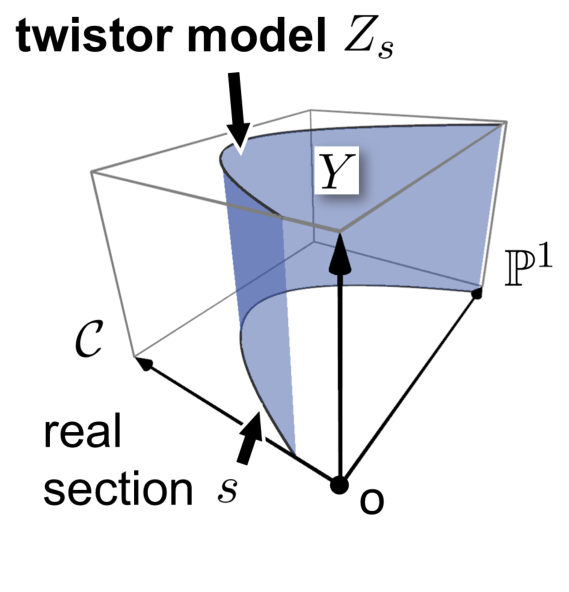}
    \caption{A conceptual diagram showing the twistor model $Z_s$ associated with a real section $s$ embedded inside the principal twistor model $\mathcal{Y}(1)$.}
    \label{fig:twistor model in PTM}
\end{minipage}
\end{figure}

\FloatBarrier
\subsection{Universality of the Principal Twistor Model}

Let $X$ be a conical symplectic variety. Assume that the regular locus $X_{\mathrm{reg}}$ admits an algebraic hyperkähler cone metric $g_0$. Let $Y$ be a crepant resolution of $X$.
In this section, we prove the universality of the principal twistor model, which states that the twistor space corresponding to any algebraic hyperkähler metric on the resolution $Y$ asymptotic to the metric $g_0$ is obtained by slicing the principal twistor model.
To this end, we first show that the algebraic hyperkähler cone metric $g_0$ naturally induces the good triple on $X$. Next, we prove the aforementioned universality for the principal twistor model constructed from this good triple.

\subsubsection{Twistor Cones and Hyperkähler Cone Metrics} \label{subsubsec:good triple from HK cone metric}

Let $X$ be a conical symplectic variety, and assume that its regular locus $X_{\mathrm{reg}}$ admits an algebraic hyperkähler cone metric $g_0$.
In this subsection, we show that this metric $g_0$ naturally induces the good triple on $X$, and that its corresponding twistor space is isomorphic to the twistor cone $X(1)$.

\medskip
First, we introduce the definition of an algebraic twistor model.
\begin{definition}\label{def:algebraic twistor model}
We say that a twistor model $(Z\to\mathbb{P}^1,\omega,\tau)$ is \textit{algebraic} if it satisfies the following three conditions:
Suppose $\mathbb{P}^1=\mathbb{C}^+\cup\mathbb{C}^-$, and let $Z^\pm\to\mathbb{C}^\pm$ denote the restrictions of the fibration $Z$ to $\mathbb{C}^{\pm}$.
\begin{enumerate}
   \item The pair $(Z^\pm, \omega^\pm)$ is an algebraic Poisson deformation, where $\omega^\pm$ denotes the restriction of the relative $2$-form $\omega$ to $Z^\pm$.
   \item The gluing function of $Z^\pm$ is algebraic.
   \item The restriction $\tau^\pm:Z^\pm\to Z^\mp$ of the real structure $\tau$ is an algebraic map.
\end{enumerate}
\end{definition}

We similarly define the notion of being algebraic for twistor spaces and corresponding hyperkähler metrics:
\begin{definition}\label{def:alg for tw and HK met}
   We say that a twistor space is \textit{algebraic} if it is algebraic as a twistor model.
   We call a hyperkähler metric corresponding to an algebraic twistor space an \textit{algebraic hyperkähler metric}.
\end{definition}

\begin{example}\label{ex:alg HK metric}
   We list a few examples and a non-example of algebraic hyperkähler metrics to clarify the definitions:
   \begin{enumerate}
    \setlength{\emergencystretch}{2em}
      \item \textit{ALE gravitational instantons}:~These are typical examples of algebraic hyper\-kähler metrics (cf.~\cite{Hitchin79}).
      \item \textit{ALF gravitational instantons}:~These spaces provide a non-example. The twistor construction for ALF metrics involves transcendental gluing functions (e.g., exponential terms), which fails to satisfy condition (2) of Definition~\ref{def:algebraic twistor model}. Thus, they are not algebraic.
      \item \textit{Quiver varieties and toric hyperkähler manifolds}:~Hyperkähler metrics on Nakajima quiver varieties (cf.~\cite{Nakajima94}) and toric hyperkähler manifolds (cf.~\cite{Bielawski_Dancer00}) are also algebraic under several conditions (cf.~\S\ref{subsec:Metrics on HK quot}).
   \end{enumerate}
\end{example}

Next, we introduce the definition of a cone metric on a conical symplectic variety.
\begin{definition}\label{def:cone met on coni symp vrt}
   Let $X$ be a conical symplectic variety with $\mathbb{C}^*$-action $\lambda$. Assume that the regular locus $X_{\mathrm{reg}}$ admits a metric $g_0$.
   We say that the metric $g_0$ is a \textit{cone metric} on $X_{\mathrm{reg}}$ if there exists a positive integer $k$ such that $\lambda_r^*g_0 = r^k g_0$ for any action $\lambda_r$ with $r>0$.
\end{definition}

The goal of this subsection is to prove the following proposition.
\begin{proposition}\label{prop:good triple from HK cone metric}
   Let $X$ be a conical symplectic variety. Assume that the regular locus $X_{\mathrm{reg}}$ admits an algebraic hyperkähler cone metric $g_0$, and let $(Z\to\mathbb{P}^1,\omega,\tau, \{\ell_x\}_{x\in X_{\mathrm{reg}}})$ be the corresponding algebraic twistor space.
   Let $\omega_0$ be the restriction of the relative holomorphic symplectic form $\omega$ to the central fiber $Z_0 \simeq X$.
   Note that the real structure $\tau$ restricts to an anti-holomorphic map $\tau|_{Z_0}: Z_0 \to \overline{Z}_\infty$. Let $\tau_0$ be the anti-holomorphic map on $X$ induced by this restriction via the natural identifications $Z_\infty \simeq X$.
   Denoting the $\mathbb{C}^*$-action on $X$ by $\lambda_0$, the tuple $(\lambda_0, \omega_0, \tau_0)$ forms a good triple on $X$.
   Moreover, the twistor space $Z$ is isomorphic as a twistor model to the twistor cone $(X(1)\to\mathbb{P}^1,\hat{\omega}_0,\hat{\tau}_0)$ constructed from this good triple (cf.~Def.~\ref{def:twistor cone}).
\end{proposition}

\begin{remark} \label{rem:extension_from_reg}
    In Proposition \ref{prop:good triple from HK cone metric}, the geometric data induced by the metric $g_0$ (such as the anti-holomorphic map $\tau_0$) are initially defined only on the regular locus $X_{\mathrm{reg}}$.
    However, since $X$ is a normal variety and the singular locus has codimension at least $2$, any (anti-)holomorphic map defined on $X_{\mathrm{reg}}$ uniquely extends to the whole of $X$.
    Thus, the map $\tau_0$ naturally extends to an anti-holomorphic map on $X$, making $(\lambda_0, \omega_0, \tau_0)$ a well-defined triple on the entire space $X$.
\end{remark}

\begin{remark}\label{remark:weight of HK cone metric}
   By the above proposition, it follows that the hyperkähler cone metric $g_0$ satisfies $(\lambda_0)_r^*g_0 = r^2 g_0$.
\end{remark}

We need the following lemma, which states that the universal Poisson deformation of $X$ is also universal among $\mathbb{C}^*$-equivariant Poisson deformations.
\begin{lemma}[{\cite[Cor.2.4]{namikawa2026notes}}]\label{lemma: C^* equiv embed to Xcal}
   Let $X$ be a conical symplectic variety, and let $\mathcal{X}\to\mathcal{B}$ be its universal Poisson deformation. By Theorem~\ref{Namikawa Thm: C^*-action on Yuni}, $\mathcal{X}$ admits a natural $\mathbb{C}^*$-action. Consider a Poisson deformation $\mathcal{Z}\to \mathcal{D}$ of $X$ equipped with a $\mathbb{C}^*$-action that extends the $\mathbb{C}^*$-action on the central fiber $X$.
   Then, there exists a unique $\mathbb{C}^*$-equivariant map $s:\mathcal{D}\to\mathcal{B}$ between the base spaces such that there is a $\mathbb{C}^*$-equivariant isomorphism from the deformation $\mathcal{Z}$ to the pullback $s^*\mathcal{X}$ of the deformation $\mathcal{X}$ by the map $s$.
\end{lemma}

\begin{remark}\label{remark: R^+-embed to Xcal}
   Since higher group cohomology with respect to rational representations of $\mathbb{R}^+$ also vanishes, the assertion on $\mathbb{C}^*$-actions in Lemma \ref{lemma: C^* equiv embed to Xcal} holds true even if replaced with $\mathbb{R}^+$-actions.
\end{remark}

We prove the following lemma on the structure as a holomorphic fiber bundle of the twistor space corresponding to a hyperkähler cone metric on $X$.

\begin{lemma}\label{lemma: loc trivial for tw sp of HK cone}
  Let $X$ be a conical symplectic variety. Assume that the regular locus $X_{\mathrm{reg}}$ admits an algebraic hyperkähler cone metric $g_0$, and let $(Z\to\mathbb{P}^1,\omega,\tau,\allowbreak \{\ell_x\}_{x\in X_{\mathrm{reg}}})$ be the corresponding algebraic twistor space.
   Let $\mathbb{P}^1=\mathbb{C}^+\cup\mathbb{C}^-$, and let $Z^\pm\to\mathbb{C}^\pm$ be the restrictions of the fibration $Z$ to $\mathbb{C}^{\pm}$.
   Then, the Poisson deformation $Z^\pm$ on $\mathbb{C}^\pm$ is isomorphic to the trivial Poisson deformation $X\times\mathbb{C}$.
\end{lemma}

\begin{proof}
   Since the metric $g_0$ is preserved under the $\mathbb{R}^+$-action, the $\mathbb{R}^+$-action acts tri-holomorphically on the hyperkähler variety $(X,g_0)$. Therefore, the corresponding twistor space $Z$ admits a fiberwise holomorphic $\mathbb{R}^+$-action.
   Since $Z^+\to\mathbb{C}$ is a Poisson deformation, letting $\mathcal{X}\to\mathcal{B}$ be the universal Poisson deformation of $X$, by Lemma \ref{lemma: C^* equiv embed to Xcal} and Remark~\ref{remark: R^+-embed to Xcal}, there exists a unique $\mathbb{R}^+$-equivariant map $s:\mathbb{C}\to\mathcal{B}$ determining an $\mathbb{R}^+$-equivariant isomorphism $Z^+\simeq s^*\mathcal{X}$.
   Since the $\mathbb{R}^+\subset \mathbb{C}^*$-action acts with positive weight on the base space $\mathcal{B}$, if $Z^+$ admits a fiberwise $\mathbb{R}^+$-action, then it follows that $s \equiv 0$.
   That is, $Z^+$ is isomorphic to the trivial Poisson deformation $X\times\mathbb{C}$. The same holds for $Z^-$.
\end{proof}

With these preparations, we prove Proposition~\ref{prop:good triple from HK cone metric}.

\begin{proof}[Proof of Proposition~\ref{prop:good triple from HK cone metric}]
   The real structure $\tau$ on the twistor space $Z$ naturally induces a Poisson isomorphism $\tau^+:Z^+\to \overline{Z}^-$.
   Let $\psi:Z^+\setminus Z_0 \to Z^- \setminus Z_\infty$ be the gluing function between $Z^+$ and $Z^-$.
   Consider the holomorphic fiber bundle $X(1)\to\mathbb{P}^1$ obtained by applying the $\mathbb{C}^*$-gluing construction to the $\mathbb{C}^*$-action $\lambda_0$ on the conical symplectic variety $(X,\omega_0)$.
   As defined in the proposition, let $\tau_0:X\to \overline{X}$ be the anti-holomorphic map induced by the restriction $\tau|_{Z_0}: Z_0 \to \overline{Z}_\infty$ via the natural identifications $Z_0 \simeq X \simeq Z_\infty$.
   Consider the anti-holomorphic map $\tau'\coloneqq\hat{\tau}_0$ induced on the fiber bundle $X(1)$ by Proposition~\ref{prop:anti-hol map on M(1)}. Let $\psi'$ be the gluing function \eqref{gluing fn} of the fiber bundle $X(1)$.
   In what follows, under the isomorphism $Z^\pm \simeq X\times\mathbb{C}$ by Lemma \ref{lemma: loc trivial for tw sp of HK cone}, we proceed with the discussion by regarding the above maps as $\tau^+,{\tau'}^+:X\times\mathbb{C}\to \overline{X\times\mathbb{C}}$ and $\psi,\psi':X\times\mathbb{C}^*\to X\times \mathbb{C}^*$.
   Moreover, as the $\mathbb{C}^*$-action on $X\times\mathbb{C}$, we consider the fiberwise $\mathbb{C}^*$-action naturally induced by the $\mathbb{C}^*$-action on $X$.
   If we show that the two anti-holomorphic maps $\tau^+,{\tau'}^+$ and the two gluing functions $\psi,\psi'$ are both $\mathbb{C}^*$-equivariantly conjugate as Poisson isomorphisms, then we can prove the assertion of the proposition (proven in Assertion (4) below).
   We prove the following four assertions in order.
    \medskip
  \begin{enumerate}[wide=\parindent, itemsep=\medskipamount]
   \item \textit{The identity $(\tau^+)^4=\mathrm{id}_{X\times\mathbb{C}}$}:
   For any point $x\in X_{\mathrm{reg}}$, a twistor line $\ell_x$ passing through $x$ is given. Since the real structure $\tau$ preserves the twistor lines, its differential $\tau_*$ induces a real structure on the normal bundle $N_{\ell_x/Z}\simeq\mathbb{C}^{2n}\otimes\mathcal{O}(1)$. Moreover, one can verify that the real structure on the vector bundle $\mathbb{C}^{2n}\otimes\mathcal{O}(1)$ is conjugate to the real structure obtained from the standard quaternionic structure on $\mathbb{C}^{2n}$ (cf.~Example~\ref{ex:real str on O(1)+O(1)}).
   Therefore, since the differential $\tau^+_*$ satisfies $(\tau^+_*)^4=\mathrm{id}$, we have $(\tau^+)^4=\mathrm{id}_{X\times\mathbb{C}}$.

   \item \textit{The anti-holomorphic maps $\tau^+$ and ${\tau'}^+$ are conjugate}:
   The two maps $\tau^+$ and ${\tau'}^+$ are Poisson isomorphisms, coincide with the map $\tau_0$ on the central fiber, and coincide with the map $u\mapsto -\bar{u}$ on the base space $\mathbb{C}$. Moreover, they satisfy $(\tau^+)^4=({\tau'}^+)^4=\mathrm{id}_{X\times\mathbb{C}}$.
   Consider the Poisson automorphisms $f,f'$ on $X\times \mathbb{C}$ obtained by composing the two maps $\tau^+,{\tau'}^+$ with the complex conjugation on $X\times \mathbb{C}$. By Lemma \ref{lemma:uniqueness of f_n^m} and Remark~\ref{remark:uniqueness of f_n^m, C^*-equiv ver}, the morphisms induced by the two morphisms $f,f'$ on the formal Poisson deformation of $X\times \mathbb{C}$ are $\mathbb{C}^*$-equivariantly conjugate.
   Therefore, by considering their algebraization (cf.~Cor.~\ref{cor:algebraization of quat str}), it is shown that the two morphisms $f, f'$ on $X\times \mathbb{C}$ are $\mathbb{C}^*$-equivariantly conjugate.
   From this, it follows that the two morphisms $\tau^+$ and ${\tau'}^+$ are $\mathbb{C}^*$-equivariantly conjugate.

   \item \textit{The gluing functions $\psi$ and $\psi'$ coincide}:
   The gluing function $\psi$ commutes with the fiberwise $\mathbb{C}^*$-action on the Poisson deformation $(X\times \mathbb{C}^*, \hat{\omega}_0)$ and satisfies $\psi^*\hat{\omega}_0=u^{-2}\hat{\omega}_0$ for the relative holomorphic symplectic form $\hat{\omega}_0$ determined by $\omega_0$ (cf.~Prop.~\ref{prop:diff form on M(1)}).
   Therefore, for a coordinate system $x=(x_i)_{i=1}^N$ of the affine variety $X\subset\mathbb{C}^N$ on which the $\mathbb{C}^*$-action is diagonal, there exists a tuple of $N$ integers $(k_i)_{i=1}^N$ such that $\psi$ can be written in the form
   \begin{equation}
      \psi(x,u)=\left(\left(\frac{1}{u^{k_i}}x_i\right)_{i=1}^N, \frac{1}{u}\right). \label{map:gluing fn almost form}
   \end{equation}
   In particular, $\psi$ is the identity on the fiber over $u=1$.
   On the other hand, the gluing function $\psi'$ of the fiber bundle $X(1)$ is also the identity on the fiber over $u=1$ and coincides with $\psi$ on the base space by definition.
   Moreover, since $\tau^+$ is a morphism determined as the restriction of the real structure $\tau$ on $Z$, the map $\varepsilon\coloneqq\psi^{-1}\circ \tau^+$ defines a real structure on $Z^+\setminus Z_0\simeq X\times\mathbb{C}^*$. Thus, in particular, it satisfies $\varepsilon^4=\mathrm{id}_{X\times\mathbb{C}^*}$.
   On the other hand, since $\tau_0^4=\mathrm{id}_X$, the map $\varepsilon'\coloneqq{\psi'}^{-1}\circ {\tau'}^+$ also satisfies $(\varepsilon')^4=\mathrm{id}_{X\times\mathbb{C}^*}$.
    Since $\tau^+$ and ${\tau'}^+$ are $\mathbb{C}^*$-equivariantly conjugate, an argument similar to that for Assertion (2) implies that $\varepsilon$ and $\varepsilon'$ are conjugate on a neighborhood of $1 \in \mathbb{C}^*$ in the base space.
   Since we know that the map $\psi$ can be written globally in the form \eqref{map:gluing fn almost form}, we have $\psi=\psi'$.

   \item \textit{Proof of the proposition}:
   By Assertion (3), we have $Z\simeq X(1)$ as holomorphic fiber bundles. Therefore, the real structure $\tau$ of $Z$ can be regarded as a real structure on the fiber bundle $X(1)$, and by Assertion (2), the two anti-holomorphic maps $\tau$ and $\tau'$ are $\mathbb{C}^*$-equivariantly conjugate. Therefore, since the anti-holomorphic map $\tau'=\hat{\tau}_0$ is a real structure, by Proposition~\ref{prop:real str on M(1)}, we deduce that the map $\tau_0$ defines a quaternionic structure on $X$.
   Moreover, from the relation $\psi^*\hat{\omega}_0=u^{-2}\hat{\omega}_0$, we have $(\lambda_0)_c^*\omega_0 = c^2\omega_0$ for any $c\in\mathbb{C}^*$.
   Therefore, the tuple $(\lambda_0,\omega_0,\tau_0)$ forms a good triple on $X$.
   Furthermore, from the discussion so far, it is also deduced that the twistor space $Z$ is isomorphic as a twistor model to the twistor cone $X(1)$ determined by this good triple. Consequently, the assertion of the proposition follows. \qedhere
  \end{enumerate}
\end{proof}

\subsubsection{Universality of the Principal Twistor Model}

Let $X$ be a conical symplectic variety. Assume that the regular locus $X_{\mathrm{reg}}$ admits an algebraic hyper\-kähler cone metric $g_0$. Let $Y$ be a crepant resolution of $X$.
In this subsection, we prove the universality of the principal twistor model, which states that the algebraic twistor space corresponding to any hyperkähler metric on the resolution $Y$ asymptotic to $g_0$ is obtained by slicing the principal twistor model.

\medskip
First, we define the asymptotic behavior of a metric on the resolution $Y$ with respect to the cone metric.
\begin{definition}\label{def:asymp to cone metric}
   Let $X$ be a conical symplectic variety. Assume that $X_{\mathrm{reg}}$ admits a cone metric $g_0$.
   Let $Y$ be a crepant resolution of $X$.
   We say that a metric $g$ on the resolution $Y$ is \textit{asymptotic to the metric $g_0$ at infinity on $Y$} if it satisfies the following condition:
   Identifying the regular locus $X_{\mathrm{reg}}$ with a subset of $Y$ naturally, fix an arbitrary point $x_0\in X_{\mathrm{reg}}$. Then, we have
   \[\lim_{r\rightarrow \infty}||g-g_0||_{g_0(r\cdot x_0)}=0, \]
   where $r\cdot x_0$ denotes the action of $r>0$ on $X$, and $||\cdot||_{g_0(r\cdot x_0)}$ denotes the norm induced by the metric $g_0$ at the point $r\cdot x_0\in X_{\mathrm{reg}}$.
\end{definition}

\begin{remark}
   An assumption on the rate of convergence is not required.
\end{remark}

We define an algebraic asymptotic hyperkähler metric on the resolution $Y$.

\begin{definition}\label{def:alg asymp HK on Y}
Let $X$ be a conical symplectic variety, and let $g_0$ be a cone metric on its regular locus $X_{\mathrm{reg}}$. Let $Y$ be a crepant resolution of $X$.
We say that a hyperkähler metric $g$ on $Y$ is an \textit{algebraic hyperkähler metric asymptotic to $g_0$} (or simply an \textit{algebraic asymptotic hyperkähler metric}) if it satisfies the following three conditions:
\begin{enumerate}
    \item The metric $g$ is asymptotic to the cone metric $g_0$ at infinity on $Y$ in the sense of Definition \ref{def:asymp to cone metric}.
    \item The corresponding twistor space $Z \to \mathbb{P}^1$ is algebraic in the sense of Definition \ref{def:alg for tw and HK met}.
    \item There exists a Poisson deformation $\mathcal{Y} \to T$ of $Y$ such that the central fiber $Z_0$ of the twistor space $Z$ is isomorphic as Poisson variety to a fiber $\mathcal{Y}_t$ for some $t \in T$.
\end{enumerate}
\end{definition}

\begin{remark}
The first two conditions alone do not relate the twistor space to the natural holomorphic Poisson structure of $Y$.
Since we fix $Y$ as a complex symplectic variety, it is crucial to anchor the third condition specifically to the central fiber $Z_0$ (which corresponds to the complex structure $I$). This ensures that $Z_0$ and $Y$ are connected by a Poisson deformation, which is essential for applying the algebraic deformation theory of $Y$ to the study of hyperkähler metrics.
\end{remark}

We now prove the following theorem.

\begin{theorem} \label{main thm: universality of PTM}
Let $X$ be a conical symplectic variety. Assume that the regular locus $X_{\mathrm{reg}}$ admits an algebraic hyperkähler cone metric $g_0$, and that $X$ has a crepant resolution $Y$.
By Proposition~\ref{prop:good triple from HK cone metric}, the metric $g_0$ induces the good triple on $X$. For this good triple, consider the principal twistor model $(\mathcal{Y}(1)\to\mathcal{C}(2), \omega_P, \tau_P)$ constructed in Proposition~\ref{main prop: PTM of Y}.
Consider an algebraic hyperkähler metric $g$ on $Y$ that is asymptotic to the metric $g_0$ at infinity on $Y$.
Then, there exists a unique real section $s$ of the vector bundle $\mathcal{C}(2)$ such that the twistor space $Z$ corresponding to the metric $g$ is isomorphic as a twistor model to the model $Z_s$ obtained by slicing the principal twistor model $\mathcal{Y}(1)$ along the real section $s$ by Proposition~\ref{prop:slicing of PT}.

\end{theorem}

\begin{proof}
   Let $(Z\to\mathbb{P}^1,\omega,\tau)$ be the twistor model determined from the twistor space corresponding to the metric $g$.
   Let $\mathbb{P}^1=\mathbb{C}^+\cup\mathbb{C}^-$, and let $Z^\pm\to\mathbb{C}^\pm$ be the restrictions of the fibration $Z$ to $\mathbb{C}^{\pm}$.
   We prove the following three assertions in order.

  \begin{enumerate}[wide=\parindent, itemsep=\medskipamount]
   \item \textit{A real section $s$ of the vector bundle $\mathcal{C}(2)$ is determined from the twistor space $Z$}:
   By condition (3) in Definition \ref{def:alg asymp HK on Y}, the central fiber $Z_0$ is isomorphic to a fiber of the universal Poisson deformation $\mathcal{Y}\to\mathcal{C}$ of $Y$ as a Poisson variety.
   Since $Z^\pm$ is a Poisson deformation, for the universality of $\mathcal{Y}\to\mathcal{C}$, there exists a unique morphism $s^\pm:\mathbb{C}^\pm\to\mathcal{C}$ between the base spaces such that $Z^\pm \simeq (s^\pm)^*\mathcal{Y}$ as Poisson deformations.
   Moreover, letting $\psi$ be the gluing function between the restrictions of $Z^\pm$ to the base space $\mathbb{C}^*$, $\psi$ satisfies $\psi^*\omega = u^{-2}\omega$.
   From this, using the universality of the universal Poisson deformation, it follows that $u^{-2}s^+(u)=s^-(v)$ holds, where $v=u^{-1}$.
   Therefore, a holomorphic section $s$ of the vector bundle $\mathcal{C}(2)$ is determined.
   Letting $Z_s$ be the twistor model obtained by slicing the principal twistor model $\mathcal{Y}(1)\to\mathcal{C}(2)$ along the real section $s$, we have $Z^\pm \simeq Z_s^\pm$ as Poisson deformations.
   Moreover, for the real structure $\tau$, using the universality of the universal Poisson deformation, one can verify that the section $s$ is compatible with the real structure $\sigma_2$ on $\mathcal{C}(2)$ (cf.~Example~\ref{ex:quat str on O(k)}).

   \item \textit{The singular models $Z', Z'_s$ of the twistor models $Z, Z_s$ are isomorphic}:
   By Proposition~\ref{prop:good triple from HK cone metric}, the algebraic hyperkähler cone metric $g_0$ on $X$ induces the good triple $(\lambda, \omega, j)$ on $X$. Consider the principal twistor model $(\mathcal{Y}(1)\to\mathcal{C}(2),\omega_P,\tau_P)$ constructed from this good triple in Proposition~\ref{main prop: PTM of Y}.
   Consider the affinization $Z'^\pm$ of the Poisson deformation $Z^\pm$, and consider the singular twistor model $Z'$ obtained by gluing these.
   For the singular model $\mathcal{X}'(1)\to\mathcal{C}(2)$ of the principal twistor model $\mathcal{Y}(1)$ (cf.~Remark~\ref{remark:sing PTM of X}), let $Z'_s$ be the singular twistor model obtained by slicing $\mathcal{X}'(1)$ along the real section $s$ of $\mathcal{C}(2)$.
   From the choice of the real section $s$, we have $Z'^\pm\simeq Z'^\pm_s$ as Poisson deformations.
   In what follows, we show $Z'\simeq Z'_s$ as twistor models.
   Letting $(X(1)\to\mathbb{P}^1,\hat{\omega},\hat{j})$ be the twistor cone determined by the good triple on $X$, by Proposition~\ref{prop:good triple from HK cone metric}, the twistor cone $X(1)$ is isomorphic to the twistor model determined from the twistor space corresponding to the metric $g_0$.
   Now, both Poisson deformations $Z'^\pm$ and $X\times\mathbb{C}^\pm$ can be naturally embedded into the affine variety $\mathcal{X}'\times\mathbb{C}^\pm$.
   Since the metric $g$ on $Y$ is asymptotic to the metric $g_0$ at infinity on $Y$, the gluing function $\psi'$ between $Z'^\pm$ and the real structure $\tau'$ can be written in the following forms, respectively:
   in terms of coordinates $(x,u)$ on the affine variety $\mathcal{X}'\times\mathbb{C}^*$,
  \begin{equation*}
    \left\{
      \begin{aligned}
         \psi'(x,u) &= \left(\dfrac{1}{u}\cdot x + f(x,u), \dfrac{1}{u}\right), \\
         \tau'(x,u) &= \left( -\dfrac{1}{\bar{u}}\cdot \tilde{j}'(x) + \overline{g(x,u)}, -\dfrac{1}{\bar{u}}\right).
      \end{aligned}
    \right.
  \end{equation*}
    Here, we denote by $\lambda \cdot x$ and $\tilde{j}'$ the natural $\mathbb{C}^*$-action and the quaternionic structure on the deformation $\mathcal{X}'$, respectively, both of which are induced from the good triple on $X$.
  Moreover, $f$ and $g$ are holomorphic maps on $\mathcal{X}'\times\mathbb{C}^*$ that satisfy $f(x,u)\to 0$ and $g(x,u)\to 0$ as $|x|\to \infty$. Therefore, we must have $f \equiv 0$ and $g\equiv 0$.
  This shows that the gluing function $\psi'$ between $Z'^\pm$ and the real structure $\tau'$ coincide with the gluing function and the real structure of the singular twistor model $Z'_s$, respectively. That is, we have $Z'\simeq Z'_s$ as singular twistor models.

   \item \textit{The twistor models $Z, Z_s$ are isomorphic}:
   The twistor models $Z, Z_s$ are simultaneous resolutions of the model $Z'\simeq Z_s'$.
   Note that the exceptional set of the resolution $Z\to Z'$ has codimension at least $2$.
   Therefore, by Hartogs' extension theorem, we have $Z\simeq Z_s$ as twistor models. \qedhere
  \end{enumerate}
\end{proof}

\section{Applications}\label{sec:Applications}

In this section, we apply the universality of the principal twistor model (Theorem~\ref{main thm: universality of PTM}) to study the moduli space of algebraic hyperkähler structures with asymptotic behavior.
We also introduce metrics constructed by hyperkähler quotients and QALE hyperkähler metrics as specific examples of metrics to which this study is applicable under certain assumptions.

\subsection{Moduli Space of Asymptotic Hyperkähler Structures}

In this subsection, we investigate the moduli space of algebraic hyperkähler structures with asymptotic behavior.
To this end, we first study the structure of the set of twistor lines on the twistor model obtained by slicing the principal twistor model.
Next, we prove an injectivity theorem, showing that the moduli space admits an inclusion into a finite-dimensional real vector space.
Finally, when $X$ has an isolated singularity, we establish a certain openness property of the twistor space in the principal twistor model, allowing us to determine the dimension of the moduli space.

\subsubsection{Structure of the Set of Twistor Lines}

We establish the following notation.
\begin{notation}\label{notation: sp of tw lines}
   Let $X$ be a conical symplectic variety. Assume that the regular locus $X_{\mathrm{reg}}$ admits an algebraic hyperkähler cone metric $g_0$.
   Let $Y$ be a crepant resolution of $X$.
   Consider the following setup and notation:
  \begin{enumerate}
   \item Let $(\mathcal{Y}(1)\to\mathcal{C}(2), \omega_P, \tau_P)$ be the principal twistor model constructed in Theorem~\ref{main thm: universality of PTM}.
   \item Let $H^0(\mathbb{P}^1,\mathcal{C}(2))^{\sigma_2}$ denote the real vector space of real sections of the vector bundle $\mathcal{C}(2)$.
   \item Let $\Gamma(\mathbb{P}^1,P)^{\tau_P}$ denote the set of real sections of the principal twistor model $P=\mathcal{Y}(1)$ with respect to the real structure $\tau_P$.
   \item Let $Z_s$ denote the twistor model obtained by slicing the principal twistor model $P$ along a real section $s\in H^0(\mathbb{P}^1,\mathcal{C}(2))^{\sigma_2}$.
   \item Let $\mathcal{T}_s$ be the set of twistor lines on the twistor model $Z_s$.
   \item Let $\mathcal{T} \subset \Gamma(\mathbb{P}^1,P)^{\tau_P}$ be the subset defined by $\mathcal{T}\coloneqq\sqcup_s\mathcal{T}_s$.
  \end{enumerate}
\end{notation}

The goal of this subsection is to investigate the structure of the set $\mathcal{T}$ of twistor lines.
To this end, we establish the following lemma on the stability of twistor lines.

\begin{lemma}[cf.~{\cite[Lemma 2.14, Theorem 2.16]{Bielawski-Foscolo21}} ]\label{lemma:stability of tw lines}
   Consider the setup in Notation~\ref{notation: sp of tw lines}.
   Then, the subset $\mathcal{T}=\sqcup_s\mathcal{T}_s \subset \Gamma(\mathbb{P}^1,P)^{\tau_P}$ is a smooth manifold of real dimension $4n+3d$,
   where $4n=\dim_\mathbb{R}X$ and $d = \dim H^2(Y;\mathbb{C})$.
\end{lemma}

\begin{proof}
   Let $\varphi(1):\mathcal{Y}(1)\to\mathcal{C}(2)$ denote the map of the principal twistor model.
   For a twistor line $\ell\in \mathcal{T}_s$, we have the following short exact sequence:
   \begin{equation}
    0\to N_{\ell/Z_s}\to N_{\ell/P}\xrightarrow{d_\ell\varphi(1)} \mathcal{C}(2)\to 0. \label{seq:exact seq of ell in P}
   \end{equation}
   Consider the long exact sequence induced by this short exact sequence (\ref{seq:exact seq of ell in P}):
   \[H^1(N_{\ell/Z_s})\to H^1(N_{\ell/P})\to H^1(\mathcal{C}(2))=0.\]
   Now, since the section $\ell$ is a twistor line, we have $N_{\ell/Z_s}\simeq\mathbb{C}^{2n}\otimes \mathcal{O}(1)$. Thus, since $H^1(N_{\ell/Z_s})=0$, we have $H^1(N_{\ell/P})=0$ by the long exact sequence.
   Therefore, by Kodaira's theorem, it follows that the set $\mathcal{T}$ is a smooth manifold.
   Moreover, we obtain another long exact sequence induced by the short exact sequence (\ref{seq:exact seq of ell in P}):
   \[0\to H^0(N_{\ell/Z_s})\to H^0(N_{\ell/P})\to H^0(\mathcal{C}(2))\to H^1(N_{\ell/Z_s})=0.\]
   Note that $\dim H^0(N_{\ell/Z_s}) =4n$, and $\dim H^0(\mathcal{C}(2)) = 3d.$
   From this, it follows that $\mathcal{T}$ is a smooth manifold of real dimension $4n+3d$. \qedhere
\end{proof}

To investigate the structure of the set $\mathcal{T}$ of twistor lines, we also use the following result.
\begin{lemma}[{P. Slodowy \cite[p.49, \S 4.2 Remark]{Slodowy80}}]\label{Slodowy lemma: C^inf trivialization of Ycal}
Let $X$ be a conical symplectic variety. Assume that $X$ has a crepant resolution $Y$. Then, the universal Poisson deformation $\mathcal{Y}\to\mathcal{C}$ of $Y$ admits a trivialization $\mathcal{Y}\simeq Y\times\mathcal{C}$ as a smooth manifold.
\end{lemma}

\begin{remark}
   Although the original statement in \cite{Slodowy80} is formulated for the Slodowy slice of a complex semisimple Lie algebra, the argument applies directly to our setting.
\end{remark}

With these preparations, we prove the following proposition.

\begin{proposition}\label{prop:str of Tw}
   Consider the setup in Notation~\ref{notation: sp of tw lines} and
   the following two maps:
  \begin{enumerate}[leftmargin=*, itemsep=\medskipamount]
   \item Consider the evaluation map $ev_0:\Gamma(\mathbb{P}^1,P)^{\tau_P}\to\mathcal{Y}$ that maps a real section $\ell\in \Gamma(\mathbb{P}^1,P)^{\tau_P}$ to its value $\ell(0)$. Fix a smooth trivialization $\mathcal{Y}\simeq Y\times\mathcal{C}$ of the universal Poisson deformation $\mathcal{Y}$ obtained by Lemma \ref{Slodowy lemma: C^inf trivialization of Ycal}. Let $pr_1$ be the projection from the direct product $Y\times\mathcal{C}$ to $Y$. Then, we define a smooth map $ev:\Gamma(\mathbb{P}^1,P)^{\tau_P}\to Y$ by $ev\coloneqq pr_1\circ ev_0$.

   \item Let $\varphi(1)_*: \Gamma(\mathbb{P}^1,P)^{\tau_P}\to H^0(\mathbb{P}^1,\mathcal{C}(2))^{\sigma_2}$ be the push-forward map that maps a real section $\ell\in \Gamma(\mathbb{P}^1,P)^{\tau_P}$ to the composition $\varphi(1)\circ \ell$ with the map $\varphi(1):P\to\mathcal{C}(2)$ of the principal twistor model.
  \end{enumerate}
  Then, the smooth map
  \[\Phi=(ev, \varphi(1)_*):\mathcal{T}\to Y\times H^0(\mathbb{P}^1,\mathcal{C}(2))^{\sigma_2}\]
  obtained by restricting the maps $ev$ and $\varphi(1)_*$ to the smooth manifold $\mathcal{T} \subset \Gamma(\mathbb{P}^1,P)^{\tau_P}$ is a locally finite covering map.

\end{proposition}

\begin{proof}
   Fix a twistor line $\ell \in \mathcal{T}$, and let $(y,s)\coloneqq\Phi(\ell)$. Then, the fiber $\Phi^{-1}(y,s)$ over the point $(y,s)$ is precisely the set of twistor lines passing through the point $\ell(0)\in Z_{s(0)}$ on the central fiber on the twistor model $Z_s$.
   Here, using the fact that twistor lines locally form a foliation of the twistor space (cf.\cite[Theorem 3.3]{HKLR}), it follows that the fiber $\Phi^{-1}(y,s)$ is a discrete set.
   In particular, since twistor lines on the twistor model $Z_s$ are described as solutions to algebraic equations, the fiber $\Phi^{-1}(y,s)$ is a finite set.
   Then, one can verify that the map $\Phi$ is a submersion.
   Moreover, by Lemma \ref{lemma:stability of tw lines}, the real dimension of the smooth manifold $\mathcal{T}$ is $4n+3d$, which coincides with the dimension of the codomain of the map $\Phi$.
   Therefore, by the inverse function theorem, the map $\Phi$ is a local diffeomorphism.
   From the above, we conclude that the map $\Phi$ defines a locally finite covering. \qedhere
\end{proof}

\begin{remark}\label{remark:ramification of Phi}
   The map $\Phi$ is not surjective in general.
\end{remark}

For the discussion in the next subsection, we introduce a natural $\mathbb{R}^+$-action on the set $\mathcal{T}$.

\begin{lemma}\label{lemma:R^+-action on Tw}
   Consider the setup in Notation~\ref{notation: sp of tw lines},
   and the fiberwise $\mathbb{C}^*$-action $\tilde{\lambda}$ (\textit{Type 2}) on the principal twistor model $\mathcal{Y}(1)$ (cf.~Prop.~\ref{prop:C^*- on M(1)}).
   Then, writing the action of multiplication by $\varepsilon$ as $\tilde{\lambda}_\varepsilon$ for any positive number $\varepsilon>0$, we have $\tilde{\lambda}_\varepsilon \circ \ell \in \mathcal{T}_{\varepsilon^2s}$ for any twistor line $\ell \in \mathcal{T}_s$.
   That is, the fiberwise $\mathbb{R}^+\subset\mathbb{C}^*$-action preserves the set $\mathcal{T}$.
\end{lemma}

\begin{proof}
   Since the $\mathbb{R}^+\subset \mathbb{C}^*$-action on the universal Poisson deformation $\mathcal{Y}\to\mathcal{C}$ acts with weight $2$ on the base space, for a twistor line $\ell \in \mathcal{T}_s$, the section $\tilde{\lambda}_\varepsilon \circ \ell$ defines a holomorphic section on the twistor model $Z_{\varepsilon^2s}$. Since the $\mathbb{R}^+$-action is holomorphic, it preserves the normal bundle.
   Moreover, the real structure $\tau_P$ on the principal twistor model $\mathcal{Y}(1)$ commutes with the fiberwise $\mathbb{R}^+$-action by definition. Therefore, for a twistor line $\ell \in \mathcal{T}_s$, the real condition
   \[\tau_P \circ \tilde{\lambda}_\varepsilon \circ \ell = \tilde{\lambda}_\varepsilon \circ \ell \circ \sigma_{\text{ap}}\]
   holds (cf.~Def.~\ref{def:twistor space}~(T4)).
   Consequently, the assertion of the lemma follows. \qedhere
\end{proof}

\subsubsection{Injectivity Theorem for the Moduli Space}\label{subsec:injectivety theorem for moduli}

First, we prove the following proposition.
\begin{proposition}\label{prop:uniqueness of tw lines}
   Consider the setup in Notation~\ref{notation: sp of tw lines}.
   Assume that there exists a family of twistor lines $\{\ell_{y,s}\}_{y\in Y}$ on the twistor model $Z_s$, so that $Z_s$ is a twistor space. If the hyperkähler metric $g_s$ on $Y$ corresponding to the twistor space $Z_s$ is asymptotic to the metric $g_0$, then the family of twistor lines $\{\ell_{y,s}\}_{y\in Y}$ is uniquely determined.
\end{proposition}

We prove the following lemma on the existence of a family of twistor spaces.

\begin{lemma}\label{lemma:family of tw sp}
Consider the setup in Notation~\ref{notation: sp of tw lines}.
Assume that there exists a family of twistor lines $\{\ell_{y,s}\}_{y\in Y}$ on the twistor model $Z_s$, so that $Z_s$ is a twistor space. Then, there exists a family of twistor lines $\{\ell_{y,\varepsilon^2s}\in \mathcal{T}_{\varepsilon^2s}\}_{y\in Y, \varepsilon >0}$ that forms a lift of the set $Y\times\{\varepsilon^2s\mid \varepsilon >0\}$ with respect to the map $\Phi$ in Lemma \ref{prop:str of Tw}, and a family of twistor spaces $\{Z_{\varepsilon^2s}\}_{\varepsilon>0}$ is obtained.
\end{lemma}

\begin{proof}
   By Lemma \ref{lemma:R^+-action on Tw}, under multiplication by $\varepsilon >0$ for the fiberwise $\mathbb{C}^*$-action $\tilde{\lambda}$ on the principal twistor model,
   we obtain the family of twistor lines $\{\tilde{\lambda_\varepsilon}\circ \ell_{y,s}\}_{y\in Y}$ on the twistor model $Z_{\varepsilon^2s}$.
   Thus, we define $\ell_{y,\varepsilon^2s}$ as the member of this family passing through the point $y\in Y$, and consider the family of twistor lines $\{\ell_{y,\varepsilon^2s}\in \mathcal{T}_{\varepsilon^2s}\}_{y\in Y, \varepsilon >0}$.
   For this family, the assertion of the lemma holds.
   \qedhere
\end{proof}

To describe the relationship between the family of twistor spaces $\{Z_{\varepsilon^2s}\}_{\varepsilon>0}$ and the asymptotic behavior, we establish the following lemma:

\begin{lemma}\label{lemma:R^+-eq trivialization wrt X_reg}
    Suppose that $X$ is a conical symplectic variety with a crepant resolution $Y$, and consider the universal Poisson deformation $\mathcal{Y}\to\mathcal{C}$ of $Y$.
    Let $f:\mathcal{Y} \to Y\times\mathcal{C}$ be the differentiable trivialization obtained by Lemma \ref{Slodowy lemma: C^inf trivialization of Ycal}.
    Denoting by $E$ the exceptional set of $Y$ and setting $\mathcal{E}:=f^{-1}(E\times\mathcal{C})$, the open submanifold $\mathcal{Y}\setminus \mathcal{E}$ admits an $\mathbb{R}^+$-equivariant differentiable trivialization $f':\mathcal{Y}\setminus \mathcal{E} \to (Y\setminus E)\times\mathcal{C}$.
    Here, the target space $(Y\setminus E)\times\mathcal{C}$ is equipped with the $\mathbb{C}^*$-action $\lambda_Y$ on $Y$ and the scalar multiplication of weight $2$ on $\mathcal{C}$.
\end{lemma}

\begin{proof}
   We identify $Y \setminus E$ with $X_{\mathrm{reg}}$ and fix an embedding of the affine variety $X$ into $\mathbb{C}^N$. Let $\Sigma_0$ be the intersection of $X_{\mathrm{reg}}$ and the unit sphere centered at the origin in $\mathbb{C}^N$. We also define $\Sigma:=f^{-1}(\Sigma_0\times\mathcal{C})\subset\mathcal{Y}\setminus \mathcal{E}$. By considering the cone $C(\Sigma)$ generated by the $\mathbb{R}^+$-action $\tilde{\lambda}_Y$ on $\mathcal{Y}$, we obtain a diffeomorphism $\mathcal{Y}\setminus \mathcal{E}\simeq C(\Sigma)$.
   Similarly, by considering the cone $C(\Sigma_0)$ generated by the $\mathbb{C}^*$-action $\lambda_Y$ on $Y$, we obtain a diffeomorphism $Y\setminus E \simeq C(\Sigma_0)$.
   Any point $y$ in the cone $C(\Sigma)$ can be written in the form $\tilde{\lambda}_{Y,r}(f^{-1}(x_0,c))$, where $r>0$, $x_0 \in \Sigma_0$, and $c\in\mathcal{C}$. By defining $f'(y):=(\lambda_{Y,r}(x_0), r^2c)$, we obtain a natural diffeomorphism
   \[f':C(\Sigma) \to C(\Sigma_0) \times\mathcal{C},\]
   which is an $\mathbb{R}^+$-equivariant map.
\end{proof}

With these preparations, we prove Proposition~\ref{prop:uniqueness of tw lines}.

\begin{proof}[Proof of Proposition~\ref{prop:uniqueness of tw lines}]
   By Lemma \ref{lemma:family of tw sp}, there exists a family of twistor lines $\{\ell_{y,\varepsilon^2s}\in \mathcal{T}_{\varepsilon^2s}\}_{y\in Y, \varepsilon >0}$ that forms a lift of the set $Y\times\{\varepsilon^2s\mid \varepsilon >0\}$ with respect to the map $\Phi$ in Lemma \ref{prop:str of Tw}.
   Consider the twistor cone $X(1)$ associated with the hyperkähler metric $g_0$ on $X_{\mathrm{reg}}$ (cf.~Prop.~\ref{prop:good triple from HK cone metric}), and let $\{\ell_{x,0}\}_{x\in X_{\mathrm{reg}}}$ be the family of twistor lines on it.
   Now, using the $\mathbb{R}^+$-equivariant differentiable trivialization $\mathcal{Y}\setminus \mathcal{E}\simeq (Y\setminus E)\times\mathcal{C}$ obtained in Lemma \ref{lemma:R^+-eq trivialization wrt X_reg}, we relabel the family of twistor lines $\{\ell_{y,\varepsilon^2s}\}_{y\in Y\setminus E, \varepsilon >0}$.
   Since each twistor line defines a section of the principal twistor model $\mathcal{Y}(1)$, fixing a point $u\in\mathbb{P}^1$ determines a point on the Poisson deformation family $\mathcal{Y}$. Therefore, under the differentiable trivialization $\mathcal{Y}\setminus \mathcal{E}\simeq (Y\setminus E)\times\mathcal{C}$, the condition that the hyperkähler metric $g_s$ on $Y$ corresponding to the twistor space $Z_s$ is asymptotic to the metric $g_0$ at infinity on $Y$ can be expressed as follows: for any point $y\in Y\setminus E$ and $u\in\mathbb{P}^1$,
   \begin{equation}
      \lim_{r\rightarrow \infty}d_0(\ell_{r\cdot y,s}(u), \ell_{r\cdot y,0}(u))=0.  \label{eq:asymp cond 1}
   \end{equation}
   Here, $d_0$ denotes the distance function induced from the cone metric $g_0$ on $Y\setminus E \simeq X_{\mathrm{reg}}$ and the natural metric on the vector space $\mathcal{C}$. In addition, $r\cdot y$ denotes the $\mathbb{R}^+$-action on $Y$.
   Furthermore, letting $\tilde{\lambda}_Y$ denote the $\mathbb{R}^+$-action on $\mathcal{Y}$, for any $r>0$ and $y\in Y\setminus E$, we have
   \begin{equation}
      \tilde{\lambda}_{Y,r}(\ell_{y,s})=\ell_{r\cdot y, r^2s}. \label{eq:R^+ on twistor lines}
   \end{equation}

   Hence, by setting $\varepsilon = r^{-1}$, the asymptotic condition \eqref{eq:asymp cond 1} can be rewritten as:
   \[d_0(\ell_{y,\varepsilon^2s}(u), \ell_{y,0}(u))=o(\varepsilon^2).\]
   Here, we used the relation $\tilde{\lambda}_{Y,r}^*d_0 = r^2d_0$.
   Therefore, considering the family of twistor lines $\{\ell_{y,\varepsilon^2s}\in \mathcal{T}_{\varepsilon^2s}\}_{y\in (Y\setminus E), \varepsilon \geq 0}$ extended to $\varepsilon \geq 0$, it follows that this family gives a lift of the set $(Y\setminus E)\times\{\varepsilon^2s\mid \varepsilon \geq 0\}$ with respect to the covering map $\Phi$ in Lemma \ref{prop:str of Tw}.
   That is, the families of twistor lines corresponding to the metrics $g_s$ and $g_0$ exist on the same sheet with respect to the covering map $\Phi$.
   From this, the assertion of the proposition follows. \qedhere
\end{proof}

\begin{remark}\label{rem:degeneration of tw sp}
   As shown in the above proof, for any $\varepsilon>0$, the hyperkähler metric $g_{\varepsilon^2s}$ on $Y$ corresponding to the twistor space $Z_{\varepsilon^2s}$ is also asymptotic to the cone metric $g_0$ at infinity. This asymptotic behavior can be understood from two perspectives:

   \begin{enumerate}
       \item \textit{Twistor geometric perspective}: The proof shows that the asymptotic behavior can be reinterpreted geometrically as the degeneration of the family of twistor spaces $\{Z_{\varepsilon^2s}\}_{\varepsilon>0}$ to the twistor cone $X(1)$ associated with the cone metric $g_0$.

       \item \textit{Hyperkähler geometric perspective}: In the language of hyperkähler metrics, this degeneration is characterized by the scaling property of the family of metrics. Under the $\mathbb{R}^+$-equivariant identification $\mathcal{Y}\setminus \mathcal{E}\simeq (Y\setminus E)\times\mathcal{C}$, the relation \eqref{eq:R^+ on twistor lines} for the twistor lines translates to the relation for the corresponding metrics:
       \begin{equation}
           \lambda_{Y,r}^*g_s = r^2 g_{r^{-2}s},
       \end{equation}
       where $\lambda_{Y}$ denotes the $\mathbb{R}^+$-action on $Y$.
       Namely, the family of metrics $\{g_{\varepsilon^2 s}\}_{\varepsilon>0}$ corresponds to the degeneration to the tangent cone at infinity.
   \end{enumerate}
\end{remark}


As a consequence of this proposition, we obtain the following injectivity theorem for the moduli space. The proof is straightforward and thus omitted.

\begin{corollary}\label{cor:inclusion of moduli space}
    Let $X$ be a conical symplectic variety such that the regular locus $X_{\mathrm{reg}}$ admits an algebraic hyperkähler cone metric $g_0$, and let $Y$ be a crepant resolution of $X$.
    Let $\mathcal{Y}(1)\to\mathcal{C}(2)$ be the principal twistor model constructed in Theorem~\ref{main thm: universality of PTM}.
    Consider the moduli space $\mathcal{M}$ of algebraic hyperkähler structures defined by
    \[
        \mathcal{M} = \left\{(Y,g,I,J,K) \mid \text{$g$ is asymptotic to $g_0$ at infinity on $Y$}\right\}/(\text{isomorphism}).
    \]
    For any representative $m \in \mathcal{M}$, let $Z$ be the corresponding twistor space.
    Then, by the same theorem, there exists a unique real section $s\in H^0(\mathbb{P}^1,\mathcal{C}(2))^{\sigma_2}$ such that $Z$ and $Z_s$ are isomorphic as twistor models.
    In this case, the map
    \[
        \Psi:\ \mathcal{M} \to H^0(\mathbb{P}^1,\mathcal{C}(2))^{\sigma_2}
    \]
    defined by $\Psi(m)=s$ is injective.
\end{corollary}

\begin{remark}
   Since $\mathcal{C}\simeq H^2(Y;\mathbb{C})$ (cf.~Theorem~\ref{Namikawa Thm:CM and univ Poisson}), the real vector space $H^0(\mathbb{P}^1,\mathcal{C}(2))^{\sigma_2}$ is isomorphic to $H^2(Y;\mathbb{R})\otimes \mathbb{R}^3.$
\end{remark}

\begin{remark}\label{rem:Torelli-type injectivity}
    Geometrically, the injectivity of the map $\Psi$ relates to the injectivity of the period map for hyperkähler structures.
    For a hyperkähler structure $m=(Y,g,I,J,K) \in \mathcal{M}$, evaluating the corresponding section $s = \Psi(m)$ at the points $0,1,i\in \mathbb{P}^1$ corresponding to the complex structures $I, J$, and $K$ yields elements in the base space $\mathcal{C}$.
    These elements correspond to the cohomology classes of the holomorphic symplectic forms on the complex manifolds $(Y,I), (Y,J)$, and $(Y,K)$, respectively (via period maps for their universal Poisson deformations).
    For instance, the holomorphic symplectic form on $(Y,I)$ is given by $\omega_J + i\omega_K$, where $\omega_I, \omega_J$, and $\omega_K$ are the associated Kähler forms.
    Through this relation, the uniqueness of the section $s$ implies that the hyperkähler metric is uniquely determined by the periods (i.e., the cohomology classes) of its three Kähler forms.
    In this sense, the injectivity of $\Psi$ can be viewed as an analogue of the injectivity part of the Torelli-type theorem by Kronheimer \cite{Kronheimer89}.
    However, fully establishing this exact correspondence, including the precise determination of the period domain, is left for future research (cf.~Question~\ref{ques:characterize the period domain}).
\end{remark}

\begin{remark}\label{rem:I fixed moduli}
   The moduli space $\mathcal{M}_I\subset \mathcal{M}$ of algebraic hyperkähler structures obtained by fixing the complex structure $I$ is obtained by imposing the condition $s(0)=0$ on the real section $s\in H^0(\mathbb{P}^1,\mathcal{C}(2))^{\sigma_2}$.
   Furthermore, when $s(0)=0$, since the real section $s$ is $\mathbb{C}^*$-equivariant, one can verify that the twistor space corresponding to any element of $\mathcal{M}_I$ admits the $\mathbb{C}^*$-action of \textit{Type 1} \eqref{eq:type 1}.

\end{remark}

We also obtain the following corollary on the moduli space of hyperkähler metrics. The proof is straightforward and thus omitted.

\begin{corollary}\label{cor:moduli space of HK met}
Consider the situation of Corollary~\ref{cor:inclusion of moduli space}. Let $\mathcal{M}$ be the moduli space of algebraic hyperkähler structures, and consider the injection $\Psi:\ \mathcal{M} \to H^0(\mathbb{P}^1,\mathcal{C}(2))^{\sigma_2}$.
Now, there exists a right action by the group $SO(3)$ corresponding to hyperkähler rotation on the moduli space $\mathcal{M}$,
and the group $\mathrm{Aut}(\mathbb{P}^1, \sigma_{\text{ap}})\simeq SO(3)$ defines a right action on the real vector space $H^0(\mathbb{P}^1,\mathcal{C}(2))^{\sigma_2}$ by pullback. Here, $\sigma_{\text{ap}}$ denotes the antipodal map on $\mathbb{P}^1$, and
\[\mathrm{Aut}(\mathbb{P}^1, \sigma_{\text{ap}})=\{f\in \mathrm{Aut}(\mathbb{P}^1)\mid \sigma_{\text{ap}}\circ f = f\circ \sigma_{\text{ap}}\}.\]
In this case, the injection $\Psi$ is equivariant with respect to these actions.
Therefore, letting $\mathcal{M}_{\text{met}}=\mathcal{M}/SO(3)$ be the moduli space of hyperkähler metrics obtained by identifying the tuples of complex structures $(I,J,K)$ that map to each other by hyperkähler rotation,
the map to the real analytic space $H^0(\mathbb{P}^1,\mathcal{C}(2))^{\sigma_2}/\mathrm{Aut}(\mathbb{P}^1, \sigma_{\text{ap}})$
\[\Psi_{\text{met}}:\mathcal{M}_{\text{met}}\to H^0(\mathbb{P}^1,\mathcal{C}(2))^{\sigma_2}/\mathrm{Aut}(\mathbb{P}^1, \sigma_{\text{ap}})\]
is injective.
\end{corollary}


\subsubsection{Openness of Twistor Space in the Principal Twistor Model}\label{subsubsec:openess of Tw sp in PTM}

Let $X$ be a conical symplectic variety. Assume that the regular locus $X_{\mathrm{reg}}$ admits an algebraic hyperkähler cone metric $g_0$.
Let $Y$ be a crepant resolution of $X$.
In this subsection, we show that a certain kind of openness of the twistor space in the principal twistor model holds when $X$ has an isolated singularity.

\medskip
We prove the following proposition.

\begin{proposition}\label{prop:openess of tw sp in PTM}
   Let $X$ be a conical symplectic variety with an isolated singularity, and assume that the regular locus $X_{\mathrm{reg}}$ admits an algebraic hyperkähler cone metric $g_0$.
   Let $Y$ be a crepant resolution of $X$.
   Let $(\mathcal{Y}(1)\to\mathcal{C}(2), \omega_P, \tau_P)$ be the principal twistor model given by Theorem~\ref{main thm: universality of PTM}, and let $H^0(\mathbb{P}^1,\mathcal{C}(2))^{\sigma_2}$ denote the real vector space of real sections of the vector bundle $\mathcal{C}(2)$.
   For a real section $s \in H^0(\mathbb{P}^1,\mathcal{C}(2))^{\sigma_2}$, let $Z_s$ be the twistor model obtained by slicing $\mathcal{Y}(1)$ along $s$.
   If $Z_s$ is the twistor space associated with a hyperkähler metric on $Y$ that is asymptotic to $g_0$ at infinity, then for any real section $s'$ in a sufficiently small open neighborhood of $s$, the model $Z_{s'}$ is also the twistor space of such a metric.
\end{proposition}

\begin{proof}
   By Proposition~\ref{prop:str of Tw}, there exists a locally finite covering map $\Phi:\mathcal{T}\to Y\times H^0(\mathbb{P}^1,\mathcal{C}(2))^{\sigma_2}$.
   We prove the following three assertions in order.
  \begin{enumerate}[wide=\parindent, itemsep=\medskipamount]
   \item \textit{Existence of a family of twistor lines derived from the cone metric $g_0$}:
   Note that since $X$ has an isolated singularity by assumption, the exceptional set $E \subset Y$ is compact.
   We show that for any relatively compact open neighborhood $U$ of the zero section $0\in H^0(\mathbb{P}^1,\mathcal{C}(2))^{\sigma_2}$, if we choose a sufficiently large compact set $K$ containing $E$, then the map $\Phi$ in Proposition~\ref{prop:str of Tw} is surjective onto the set $(Y\setminus K)\times U$.
   Let $V_0$ and $V_1$ be relatively compact open neighborhoods of $E$ such that $\overline{V_0}\subset V_1$, and consider the compact set $K_0\coloneqq\overline{V_1}\setminus V_0$.
   Under the identification $X_{\mathrm{reg}} \simeq Y \setminus E$, the regular locus of the twistor cone $X(1)$ associated with the hyperkähler cone metric $g_0$ is embedded into the model $Y(1)\subset \mathcal{Y}(1)$ (cf.~Remark~\ref{remark:Conceptual diagrams of PTM}).
   Hence, there exists an open neighborhood $U_0$ of the zero section $0\in H^0(\mathbb{P}^1,\mathcal{C}(2))^{\sigma_2}$ such that the map $\Phi$ is surjective onto the set $K_0 \times U_0$.
   Now, consider an arbitrary relatively compact open neighborhood $U$ of the zero section $0\in H^0(\mathbb{P}^1,\mathcal{C}(2))^{\sigma_2}$.
   If we choose a sufficiently large compact set $K$ containing $E$ in $Y$, then for any $y\in Y\setminus K$ and $s\in U$, there exist a real number $r>0$ and a twistor line $\ell$ such that $\Phi(\ell) \in K_0 \times U_0$ and $\Phi(\tilde{\lambda}_r\circ\ell)=(y,s)$.
   Here, $\tilde{\lambda}$ denotes the fiberwise $\mathbb{C}^*$-action on the principal twistor model $\mathcal{Y}(1)$.
   Therefore, by Lemma \ref{lemma:R^+-action on Tw}, $\tilde{\lambda}_r\circ\ell$ is also a twistor line for $r>0$, which implies that the map $\Phi$ is surjective onto the set $(Y\setminus K)\times U$.

   \item \textit{The twistor model $Z_{s'}$ is a twistor space}:
   Consider a relatively compact open set $U$ containing both the real section $s\in H^0(\mathbb{P}^1,\mathcal{C}(2))^{\sigma_2}$ and the zero section. By Assertion (1), there exists a compact set $K$ containing the exceptional set $E$ on the resolution $Y$ such that the map $\Phi$ is surjective onto the set $(Y\setminus K)\times U$.
   On the other hand, since $Z_s$ is a twistor space by assumption, for the compact set $K$ on $Y$, by choosing a sufficiently small open neighborhood $U'\subset U$ of the real section $s\in H^0(\mathbb{P}^1,\mathcal{C}(2))^{\sigma_2}$, the map $\Phi$ is surjective onto the set $K \times U'$.
   Therefore, the map $\Phi$ is surjective onto the set $Y\times U'$.
   Now, the crepant resolution $Y$ is simply connected and we may assume that the neighborhood $U'$ is also simply connected.
   This allows us to lift the set $Y\times U'$ to the finite covering $\Phi$, and we can choose the lift that contains the twistor lines of $Z_s$.
   Thus, for any real section $s'\in U'$, the twistor model $Z_{s'}$ is a twistor space.

   \item \textit{Asymptotic behavior of the metric corresponding to the twistor space $Z_{s'}$}:
   For the real section $s$, consider an open neighborhood $U'$ satisfying Assertion (2).
   Consider an arbitrary real section $s'\in U'$, and let $g_{s'}$ be the metric corresponding to the twistor space $Z_{s'}$.
   By applying an argument on twistor lines similar to the proof of Proposition~\ref{prop:uniqueness of tw lines}, one can show that the metric $g_{s'}$ is asymptotic to the cone metric $g_0$. Consequently, the assertion of the proposition follows. \qedhere
  \end{enumerate}
\end{proof}

As a corollary of this proposition, we obtain the following result on the moduli space of algebraic hyperkähler structures with asymptotic behavior.

\begin{corollary}\label{cor:moduli space of HK str, isolated sing case}
   Let $X$ be a conical symplectic variety with an isolated singularity. Assume that the regular locus $X_{\mathrm{reg}}$ admits an algebraic hyperkähler cone metric $g_0$.
   Let $Y$ be a crepant resolution of $X$.
   Then, the following two assertions hold:

   \begin{enumerate}[itemsep=\medskipamount]
      \setlength{\emergencystretch}{2em}
      \item The moduli space $\mathcal{M}$ of algebraic hyperkähler structures containing a hyperkähler metric asymptotic to the metric $g_0$ at infinity on $Y$ is embedded as an open set into the real vector space $H^0(\mathbb{P}^1,\mathcal{C}(2))^{\sigma_2}$.
      \item The moduli space $\mathcal{M}_{\text{met}}$ of hyperkähler metrics asymptotic to the metric $g_0$ at infinity on $Y$ is embedded as an open set into the real analytic space $H^0(\mathbb{P}^1,\mathcal{C}(2))^{\sigma_2}/\mathrm{Aut}(\mathbb{P}^1,\sigma_{\text{ap}})$.
   \end{enumerate}

   In particular, if the moduli space $\mathcal{M}$ \textup{(}resp.\ $\mathcal{M}_{\text{met}}$\textup{)} is non-empty, then we have $\dim_{\mathbb{R}} \mathcal{M} = 3d$ \textup{(}resp.\ $\dim_{\mathbb{R}} \mathcal{M}_{\text{met}} = 3d-3$\textup{)}, where $d=\dim H^2(Y;\mathbb{C})$.
\end{corollary}

\begin{proof}
   By Proposition~\ref{prop:openess of tw sp in PTM}, both the map $\Psi$ in Corollary~\ref{cor:inclusion of moduli space} and the map $\Psi_{\text{met}}$ in Corollary~\ref{cor:moduli space of HK met} are open maps. Since we have $\mathcal{C}\simeq H^2(Y;\mathbb{C})$ as vector spaces (cf.~Theorem~\ref{Namikawa Thm:CM and univ Poisson}), the assertion on the dimension of the moduli spaces follows. \qedhere
\end{proof}

\begin{remark}
   The dimension of the moduli $\mathcal{M}_I\subset \mathcal{M}$ of algebraic hyperkähler structures obtained by fixing the complex structure $I$ (cf.~Remark~\ref{rem:I fixed moduli}) is $\dim_\mathbb{R} \mathcal{M}_I=d$ if $\mathcal{M}_I$ is non-empty.
\end{remark}

\begin{remark}\label{rem:application to minimal nilpotent orbits}
   Corollary~\ref{cor:moduli space of HK str, isolated sing case} applies to the closures of minimal nilpotent orbits in complex semisimple Lie algebras.
   Let $X = \overline{\mathcal{O}_{\mathrm{min}}}$ be the closure of the minimal nilpotent orbit (cf.~\cite{Collingwood-McGovern93}) in a complex semisimple Lie algebra $\mathfrak{g}$. The space $X$ is a conical symplectic variety with an isolated singularity at the origin, and its regular locus $X_{\mathrm{reg}} = \mathcal{O}_{\mathrm{min}}$ admits an exact hyperkähler cone metric via Nahm's equations, as constructed by Kronheimer \cite{Kronheimer90}.
   As a specific example, consider the case $\mathfrak{g} = \mathfrak{sl}_n(\mathbb{C})$ for $n \ge 2$, where $X$ admits a crepant resolution $Y = T^*\mathbb{P}^{n-1}$.
   On this resolution, there exists the Calabi metric, which is an asymptotically conical hyperkähler metric that approaches the cone metric at infinity. In this setting, Corollary~\ref{cor:moduli space of HK str, isolated sing case} shows that the dimension of the moduli space $\mathcal{M}_{\text{met}}$ of such asymptotic metrics is exactly $3\dim H^2(T^*\mathbb{P}^{n-1};\mathbb{C}) - 3 = 0$ (meaning the metric is unique up to scaling).
\end{remark}

\subsection{Examples}

In this section, we introduce metrics constructed by hyperkähler quotients and QALE hyperkähler metrics as examples of metrics to which this study is applicable under certain assumptions.

\subsubsection{Metrics Constructed by Hyperkähler Quotients}\label{subsec:Metrics on HK quot}

We recall the hyper\-kähler moment map.
\begin{definition}\label{def:HK moment map}
Assume that a compact Lie group $G$ acts tri-holomorphically on a hyperkähler manifold $(M,g,I,J,K)$.
Let $\mathfrak{g}$ be the Lie algebra of the Lie group $G$, and let $\mathfrak{g}^*$ be its dual space.
We say that a map
\[\mu=(\mu_I,\mu_J,\mu_K):M\to\mathfrak{g}^*\otimes \mathbb{R}^3\]
is a \textit{hyperkähler moment map} if it satisfies the following two conditions:
\begin{enumerate}
   \item The map $\mu$ is $G$-equivariant with respect to the adjoint action of $G$ on $\mathfrak{g}^*$.
   \item For any $A\in\{I,J,K\}, \xi \in\mathfrak{g}$, and $v\in TM$,
   \[d\langle\mu_A,\xi\rangle (v)=\omega_A(v,\xi^*),\]
   where $\langle\cdot,\cdot\rangle$ denotes the pairing between $\mathfrak{g}$ and $\mathfrak{g}^*$, and $\omega_A\coloneqq g(A\cdot, \cdot)$.
\end{enumerate}
\end{definition}

We recall the following theorem on the hyperkähler quotient.
\begin{theorem}[{\cite[Theorem 3.2]{HKLR}}]\label{HKLR thm:HK quotient}
   Assume that a compact Lie group $G$ acts tri-holomorphically on a hyperkähler manifold $(M,g,I,J,K)$.
   Let $\mu:M\to\mathfrak{g}^*\otimes \mathbb{R}^3$ be an associated hyperkähler moment map, and denote the center of $\mathfrak{g}^*$ by $\mathfrak{z}^*$.
   For $\zeta \in \mathfrak{z}^*\otimes \mathbb{R}^3$, let
   \[X_\zeta \coloneqq \mu^{-1}(\zeta)/G.\]
   If the group $G$ acts freely on $\mu^{-1}(\zeta)$, then $X_\zeta$ has the natural structure of a hyperkähler manifold of real dimension $4(m-k)$,
   where $4m=\dim_\mathbb{R}M$ and $k=\dim_\mathbb{R}G$.
   Moreover, for the natural inclusion map $\iota:\mu^{-1}(\zeta)\hookrightarrow M$ and the quotient map $q:\mu^{-1}(\zeta)\to X_\zeta$, the hyperkähler structure $(\omega'_I,\omega'_J,\omega'_K)$ on $X_\zeta$ satisfies the following relation with respect to the hyperkähler structure $(\omega_I,\omega_J,\omega_K)$ on $M$: for any $A\in\{I,J,K\}$,
   \[\iota^*\omega_A = q^*\omega'_A.\]
\end{theorem}

\begin{remark}
   The hyperkähler manifold $X_\zeta$ in the above theorem is called the \textit{hyperkähler quotient} with respect to the hyperkähler moment map $\mu$.
\end{remark}

By applying the theory of symplectic quotients and the Kempf--Ness theorem in geometric invariant theory (GIT), we obtain the following lemma (cf.~\cite[Theorem 3.1, Theorem 4.1]{Nakajima94}). The proof is omitted.
\begin{lemma}\label{lemma:resolution map of HK quot}
    Let $G\subset Sp(n)$ be a compact Lie group acting linearly on the quaternionic vector space $M\coloneqq\mathbb{H}^n$, and let $\mu \colon M \to\mathfrak{g}^*\otimes\mathbb{R}^3$ be the standard hyperkähler moment map (with $\mu(0)=0$).
    For the natural $I$-holomorphic action of the complexification $G_\mathbb{C} \subset Sp(2n,\mathbb{C})$ on $M\simeq \mathbb{C}^{2n}$, we define the corresponding $I$-holomorphic moment map
    \[\mu_\mathbb{C}\coloneqq\mu_J+i\mu_K:M\to\mathfrak{g}_\mathbb{C}^*.\]
    Under the identification $\mathbb{R}^3 \simeq \mathbb{R}\times \mathbb{C}$, let $\zeta =(\zeta_I,\zeta_\mathbb{C}) \in \mathfrak{z}^* \otimes \mathbb{R}^3$, and let $\zeta_0\coloneqq(0,\zeta_\mathbb{C})$.
    Assume that the fiber $\mu_\mathbb{C}^{-1}(\zeta_\mathbb{C})$ has $0$-stable points (cf.\ \cite{Nakajima94}).
    Then, $X_{\zeta_0}\coloneqq\mu^{-1}(\zeta_0)/G$ is biholomorphic to the affine GIT quotient $\mu_\mathbb{C}^{-1}(\zeta_\mathbb{C})//G_\mathbb{C}$ with respect to the complex structure $I$.
    Furthermore, setting $X_\zeta\coloneqq\mu^{-1}(\zeta)/G$, there exists a natural $I$-holomorphic, proper, and birational map
    \[\pi_I:X_\zeta\to X_{\zeta_0}\]
    induced by the inclusion $\mu^{-1}(\zeta) \subset \mu_\mathbb{C}^{-1}(\zeta_\mathbb{C})$.
\end{lemma}

\begin{remark}\label{remark:properties of HK quot}
   We make the following three observations on the hyperkähler quotient:
   \begin{enumerate}
      \item By the symmetry of $(I,J,K)$, holomorphic maps $\pi_J$ and $\pi_K$ are obtained similarly.
      \item If $X_\zeta$ is non-singular, then it naturally inherits a hyperkähler structure. Since the map $\pi_I$ preserves the induced $I$-holomorphic symplectic form, it automatically gives a (not necessarily projective) crepant resolution of $X_{\zeta_0}$.
      \item Let $\mathfrak{z}_{\mathbb{Z}}^*\simeq \mathrm{Hom}(G,S^1)\subset\mathfrak{z}^*$ denote the character lattice of $G$. If $\zeta_I$ is a lattice point (i.e., $\zeta_I \in \mathfrak{z}_{\mathbb{Z}}^*$), it can be shown that $\pi_I$ is indeed a projective crepant resolution (we omit the proof here).
   \end{enumerate}
\end{remark}

The goal of this subsection is to prove the following proposition.

\begin{proposition}\label{prop:asymp of HK quotient}
    Let $G\subset Sp(n)$ be a compact Lie group acting linearly on the quaternionic vector space $M\coloneqq \mathbb{H}^n$, and let $\mu \colon M \to\mathfrak{g}^*\otimes\mathbb{R}^3$ be the standard hyperkähler moment map (with $\mu(0)=0$).
    Then, the central quotient $X_0\coloneqq\mu^{-1}(0)/G$ is a conical symplectic variety, and there exists a hyperkähler cone metric $g_0$ on its regular locus $(X_0)_{\mathrm{reg}}$.
    Assume further that the central fiber $\mu_\mathbb{C}^{-1}(0)$ has $0$-stable points (cf.\ \cite{Nakajima94}).
    Then, for a generic parameter $\zeta \in \mathfrak{z}^*\otimes \mathbb{R}^3$ such that the hyperkähler quotient $X_\zeta \coloneqq\mu^{-1}(\zeta)/G$ is non-singular, the hyperkähler metric $g_\zeta$ on $X_\zeta$ is asymptotic to the cone metric $g_0$ at infinity.
\end{proposition}

\begin{proof}
We prove the following four assertions in order.

\begin{enumerate}[wide=\parindent, itemsep=\medskipamount]
   \item \textit{$X_0$ is a conical symplectic variety}:
   Consider the $I$-holomorphic moment map $\mu_\mathbb{C}=\mu_J+i\mu_K: M\to \mathfrak{g}^*_\mathbb{C}$.
   By Lemma \ref{lemma:resolution map of HK quot}, the central quotient $X_0$ can be written using the GIT quotient as follows:
   \[X_0=\mu_\mathbb{C}^{-1}(0)//G_\mathbb{C}.\]
   Therefore, $X_0$ is an affine variety.
   Since the $\mathbb{C}^*$-action $\lambda$ on $(M,I)$ satisfies $\lambda_c^* \mu_{\mathbb{C}} = c^2 \mu_{\mathbb{C}}$ for $c\in\mathbb{C}^*$ in our setting, it preserves the fiber $\mu^{-1}_\mathbb{C}(0)$.
   Consequently, we deduce that $X_0$ is a conical symplectic variety.

   \item \textit{$(X_0)_{\mathrm{reg}}$ admits a hyperkähler cone metric $g_0$}:
   Let $g_0$ be the hyperkähler metric on $(X_0)_{\mathrm{reg}}$ determined by the hyperkähler moment map $\mu$.
   For the natural quotient map $q:\mu^{-1}(0)\to X_0$, we have $\iota^*g=q^*g_0$ on a dense open set of $\mu^{-1}(0)\subset M$~(cf.~Theorem~\ref{HKLR thm:HK quotient}).
   Since the metric $g$ is a hyperkähler cone metric, combined with Remark~\ref{remark:weight of HK cone metric}, we have $\lambda_r^*g_0 = r^2g_0\ (r>0)$.
   That is, $g_0$ is a hyperkähler cone metric.

   \item \textit{The hyperkähler quotient $X_\zeta$ is naturally diffeomorphic to a crepant resolution $Y$ of $X_0$}:
   We follow the argument of Kronheimer \cite{Kronheimer90}.
   Let $g_\zeta$ be the hyperkähler metric on the hyperkähler quotient $X_\zeta$, where we denote the point $\zeta$ by $\zeta=(\zeta_I,\zeta_J,\zeta_K) \in\mathfrak{g}^*\otimes \mathbb{R}^3$.
   Since the group $G$ acts linearly on the vector space $M$, the set of points $\xi\in\mathfrak{g}^*\otimes \mathbb{R}^3$ such that $X_\xi\coloneqq\mu^{-1}(\xi)/G$ admits singularities can be written as a union of finitely many hyperplanes in the vector space $\mathfrak{g}^*\otimes \mathbb{R}^3$. (That is, it has a wall-chamber structure; we omit the proof here.)
   Therefore, by considering the $SO(3)$-action on $\mathbb{R}^3$ by hyperkähler rotation,
   we may assume that the parameter $\zeta$ is generic so that $X_\zeta$, $X_{\zeta'}$, and $X_{\zeta''}$ are all non-singular, where we set $\zeta'\coloneqq(\zeta_I,\zeta_J,0)$ and $\zeta''\coloneqq(\zeta_I,0,0)$.
   Since these spaces are non-singular, the maps $\pi_K, \pi_J$, and $\pi_I$ obtained in Lemma \ref{lemma:resolution map of HK quot} automatically give crepant resolutions (see Remark~\ref{remark:properties of HK quot}), and we obtain a smooth map
   \begin{equation}
      \varphi:X_\zeta \xrightarrow{\pi_K} X_{\zeta'}\xrightarrow{\pi_J}X_{\zeta''}\xrightarrow{\pi_I}X_0 \label{seq:resolution wrt HK moment map}
   \end{equation}
   as their composition.
   Setting $Y\coloneqq X_{\zeta''}$, which is an $I$-holomorphic (not necessarily projective) crepant resolution of $X_0$, it follows that $X_\zeta$ is naturally diffeomorphic to the resolution $Y$.

   \item \textit{Asymptotic behavior of the metric $g_\zeta$ on the hyperkähler quotient $X_\zeta$}:
   Since
   \begin{equation*}
      \lambda_r^*\mu = r^2\mu
   \end{equation*}
   holds for any $r>0$, $X_{r^2\zeta}=\lambda_r (X_{\zeta})$ is also a hyperkähler quotient.
   Here, we are considering the real scalar multiplication $r \colon X_{r^{-2}\zeta}\to X_{\zeta}$ induced by the scalar multiplication $r \colon \mu^{-1}(r^{-2}\zeta) \to \mu^{-1}(\zeta)$ on $M$.
   Note that each $X_{r^2\zeta}$ is biholomorphic to $X_\zeta$ as a complex variety.
   Let $g_{r^2\zeta}$ be the hyperkähler metric on the hyperkähler quotient $X_{r^2\zeta}$.
   Since the map $\varphi$ in (\ref{seq:resolution wrt HK moment map}) is compatible with the $\mathbb{R}^+$-action on $M$,
   it satisfies
   \[\lambda_r^*g_\zeta = r^2 g_{r^{-2}\zeta}\]
   on $(X_0)_{\mathrm{reg}}\subset Y$.
   For any $x \in (X_0)_{\mathrm{reg}}\subset Y$ and $r>0$, the following equality holds: noting that $g_0$ is a cone metric,
   \begin{equation}
      \begin{aligned}
         ||g_{\zeta}-g_0||_{g_0(r\cdot x)} &= ||\lambda_r^*(g_{\zeta}-g_0)||_{\lambda_r^*g_0(x)}\\
         &= ||g_{r^{-2}\zeta}-g_0||_{g_0(x)}. \label{eq:iikae of metric of HK quot}
      \end{aligned}
   \end{equation}
   Now, since each $g_{r^{-2}\zeta}$ is a metric naturally determined by the hyperkähler moment map,
   the right-hand side of equation \eqref{eq:iikae of metric of HK quot} converges to $0$ as $r\to \infty$.
   Therefore, it follows that the metric $g_\zeta$ on the hyperkähler quotient $X_\zeta\simeq Y$ is asymptotic to the metric $g_0$ at infinity on $Y$ (cf.~Def.~\ref{def:asymp to cone metric}).
   Consequently, the assertion of the proposition follows. \qedhere
\end{enumerate}
\end{proof}

\begin{remark}\label{remark:example of HK quotient}
   As mentioned in Remark~\ref{remark:properties of HK quot}~(3), if we further assume that $\zeta_I$ is a lattice point (i.e., $\zeta_I \in \mathfrak{z}^*_{\mathbb{Z}}$), then the resolution $Y\coloneqq X_{\zeta''}$ is a projective crepant resolution of the central quotient $X_0$.
   In this case, by the argument on hyperkähler-quotient twistor space (we omit the detailed proof here), we can show that the hyperkähler metric $g_\zeta$ on $X_\zeta \simeq Y$ is an algebraic hyperkähler metric on the projective crepant resolution $Y$ (cf.~Def.~\ref{def:alg asymp HK on Y}).
   Combining this with Proposition~\ref{prop:asymp of HK quotient}, we conclude that the setting of our main theorem (Theorem~\ref{main thm: universality of PTM}) is fully satisfied for such a metric $g_\zeta$.
   Therefore, our main theorem can be applied to these hyperkähler quotients.\medskip

   Typical examples of such hyperkähler quotients include Nakajima quiver varieties (cf.~\cite{Nakajima94}) and toric hyperkähler manifolds (cf.~\cite{Bielawski_Dancer00}) that satisfy the assumptions of Proposition~\ref{prop:asymp of HK quotient}.
   Specifically, since these spaces are inherently constructed via linear actions of compact Lie groups on $\mathbb{H}^n$, our main theorem can be applied to them as long as the following two conditions hold: the underlying hyperkähler moment map is standard (i.e., $\mu(0)=0$, which is automatically satisfied for quiver varieties by definition) and the central fiber $\mu_\mathbb{C}^{-1}(0)$ has $0$-stable points.
\end{remark}

\subsubsection{QALE Hyperkähler Metrics}

QALE hyperkähler metrics, introduced by Joyce \cite{Joyce01} as a natural higher-dimensional generalization of ALE hyperkähler metrics, are also examples to which this study is applicable.
In this subsection, we assume the following.

\begin{assumption}\label{assumption:QALE HK is algebraic}
   Let $G<Sp(n)$ be a finite group, and let $X=\mathbb{C}^{2n}/G$ be the quotient space.
   Assume that $X$ admits a crepant resolution $Y$.
   We assume that any QALE hyperkähler metric on $Y$ (cf.~\cite[Definition 2.1]{Joyce01}), provided it exists, is necessarily algebraic.
\end{assumption}

\begin{remark}
   According to Joyce \cite[p.115]{Joyce01}, the above is proved by hypercomplex algebraic geometry, but the details are not described. Note that it is not known whether a QALE hyperkähler metric exists on an arbitrary crepant resolution $Y$.
\end{remark}

If a metric $g$ is a QALE metric on a crepant resolution $Y$, then it is asymptotic to the standard flat metric $g_0$ on $X$ at infinity on $Y$ in the sense of Definition~\ref{def:asymp to cone metric}. Therefore, under Assumption \ref{assumption:QALE HK is algebraic}, this study (Theorem~\ref{main thm: universality of PTM}) is applicable to QALE hyperkähler metrics.

Note that $X$ has an isolated singularity if and only if $\dim_\mathbb{C}Y=2$.
In this case, a QALE hyperkähler metric on $Y$ coincides with an ALE gravitational instanton.
Thus, by Corollary~\ref{cor:moduli space of HK str, isolated sing case}, we deduce that the moduli space of ALE gravitational instantons has real dimension $3d-3$, where $d=\dim H^2(Y;\mathbb{C})$.
This result also follows immediately from the result of Kronheimer \cite{Kronheimer89}.

On the other hand, when $\dim_\mathbb{C}Y\geq 4$, Corollary~\ref{cor:moduli space of HK str, isolated sing case} is not applicable because $X$ has non-isolated singularities (cf.~Question~\ref{ques:openess in case noncpt E}).

\subsection{Future Directions}

In this study, we showed that when there exists a hyperkähler metric with asymptotic behavior, the corresponding twistor space can be embedded into the principal twistor model. We consider the following question in the converse direction:
\begin{question}
    Let $X$ be a conical symplectic variety. Assume that the regular locus $X_{\mathrm{reg}}$ admits an algebraic hyperkähler cone metric $g_0$, and that $X$ admits a crepant resolution $Y$.
    By Proposition~\ref{prop:good triple from HK cone metric}, the metric $g_0$ induces the good triple on $X$. For this good triple, consider the principal twistor model $(\mathcal{Y}(1)\to\mathcal{C}(2), \omega_P, \tau_P)$ constructed in Proposition~\ref{main prop: PTM of Y}. In this case, under what conditions is the twistor model $Z_s$, obtained by slicing $\mathcal{Y}(1)$ along a real section $s$ of the vector bundle $\mathcal{C}(2)$, actually a twistor space?
    For example, when the generic fiber of $Z_s$ is an affine variety, does $Z_s$ admit a family of twistor lines that forms a foliation?
\end{question}
   This question can be reformulated as follows:
\begin{question} \label{ques:characterize the period domain}
   Characterize the image (i.e., the period domain) of the map
   \[\Psi:\mathcal{M}\to H^0(\mathbb{P}^1,\mathcal{C}(2))^{\sigma_2}\]
   introduced in Corollary~\ref{cor:inclusion of moduli space}.
\end{question}

\begin{remark}
   The central fiber of the model $Z_s$ is diffeomorphic to the crepant resolution $Y$.
   When $Z_s$ admits a family of twistor lines, the following can be said on the asymptotic behavior of the corresponding hyperkähler metric $g_s$ on $Y$:
   Assume that there exists a family of twistor lines on the model $Z_s$ that lie on the same sheet as those of the cone metric $g_0$ with respect to the local covering map $\Phi$ in Proposition~\ref{prop:str of Tw}. By tracing the proof of Proposition~\ref{prop:uniqueness of tw lines} backwards, one can show that the metric $g_s$ is asymptotic to the metric $g_0$ at infinity on $Y$.
\end{remark}

In considering the moduli space of algebraic hyperkähler structures, we investigated the openness of the twistor space in the principal twistor model (cf.~\S \ref{subsubsec:openess of Tw sp in PTM}).
In this context, the following question remains:
\begin{question}\label{ques:openess in case noncpt E}
   Let $X$ be a conical symplectic variety. Assume that the regular locus $X_{\mathrm{reg}}$ admits an algebraic hyperkähler cone metric $g_0$, and that $X$ admits a crepant resolution $Y$.
   Show that the map $\Psi:\mathcal{M}\to H^0(\mathbb{P}^1,\mathcal{C}(2))^{\sigma_2}$ introduced in Corollary~\ref{cor:inclusion of moduli space} is an open map when $X$ has non-isolated singularities.
\end{question}

\newpage

\section*{Acknowledgments}
The author would like to express gratitude to his supervisor, Professor Nobuhiro Honda, for providing regular opportunities for discussion and an environment that encourages independent research.
The author would also like to thank the members of his laboratory for beneficial discussions during seminars.
This work was supported by JST SPRING, Japan Grant Number JPMJSP2180.


\bibliography{source/references}
\bibliographystyle{jabbrv}

\end{document}